\documentclass[reqno]{amsart}  

\theoremstyle{plain}
\newtheorem{theorem}{Theorem}[section]
\newtheorem{lemma}[theorem]{Lemma}
\newtheorem{remark}[theorem]{Remark}

\newtheorem{corollary}[theorem]{Corollary}

\numberwithin{equation}{section}
\allowdisplaybreaks[1]

\theoremstyle{definition}

\theoremstyle{remark}

\newcommand{\bU}{{\mathbf U}}

\newcommand{\bt}{{\mathbf t}}

\newcommand{\cA}{{\mathcal A}}
\newcommand{\cB}{{\mathcal B}}
\newcommand{\cC}{{\mathcal C}}
\newcommand{\cD}{{\mathcal D}}
\newcommand{\cE}{{\mathcal E}}

\newcommand{\cG}{{\mathcal G}}
\newcommand{\cH}{{\mathcal H}}

\newcommand{\cL}{{\mathcal L}}

\newcommand{\cP}{{\mathcal P}}
\newcommand{\cR}{{\mathcal R}}
\newcommand{\cK}{{\mathcal K}}
\newcommand{\cV}{{\mathcal V}}

\newcommand{\cS}{{\mathcal S}}

\newcommand{\cU}{{\mathcal U}}
\newcommand{\cX}{{\mathcal X}}
\newcommand{\cY}{{\mathcal Y}}

\newcommand{\C}{{\mathbb C}}

\newcommand{\bx}{{\mathbf x}}
\newcommand{\bphi}{{\boldsymbol \phi}}
\newcommand{\bw}{{\mathbf w}}

\newcommand{\R}{\text{\rm Re }}
\newcommand{\I}{\text{\rm Im }}

\newcommand{\ev}{{\text{\rm \textbf{ev}}}}

\newcommand{\bcH}{{\boldsymbol{\mathcal H}}}

\begin{document}

\title[Test functions and transfer-function realization]{Test 
functions, Schur-Agler classes and transfer-function 
realizations:  the matrix-valued setting}
\author[J.A.~Ball]{Joseph A. Ball}
\address{Department of Mathematics,
Virginia Tech,
Blacksburg, VA 24061-0123, USA}
\email{joball@math.vt.edu}
\author[M.D.~Guerra-Huam\'an]{Mois\'es D.~Guerra Huam\'an}
\address{Department of Mathematics,
Virginia Tech,
Blacksburg, VA 24061-0123, USA}
\email{moisesgg@math.vt.edu}

\begin{abstract}
Given a collection of test functions, one defines the associated 
Schur-Agler class as the intersection of the contractive 
multipliers over the collection of all positive kernels for which 
each test function is a contractive multiplier.  We indicate extensions of this framework
to the case where the test functions, kernel functions, and  
Schur-Agler-class functions are allowed to be matrix- or operator-valued.
We illustrate the general theory with two examples:    (1) the matrix-valued Schur class 
over a finitely-connected planar domain and (2) the matrix-valued version of the constrained Hardy 
algebra  (bounded analytic functions on the unit disk with derivative at the origin 
constrained to have zero value).  Emphasis is on examples where the 
matrix-valued version is not obtained as a simple tensoring with 
${\mathbb C}^{N}$ of the scalar-valued version.
\end{abstract}



\maketitle

\section{Introduction}  \label{S:Intro}

In honor of the work of Issai Schur (see \cite{Schur}), it is common 
nowadays to refer to the class of holomorphic functions $s$ mapping 
the unit disk ${\mathbb D}$ into the closed unit disk 
$\overline{\mathbb D}$ as the {\em Schur class} ${\mathcal S}$. 
We summarize some of the many characterizations of the Schur class
in the following theorem.

\begin{theorem} \label{T:classicalSchur}
For a given $s \colon {\mathbb D} \to {\mathbb C}$, 
the following are equivalent:
\begin{enumerate}
    \item $s \in {\mathcal S}$,
    \item the de Branges-Rovnyak kernel associated with $s$ is a 
    positive kernel on ${\mathbb D}$:
    \begin{equation}  \label{IdeBRker}
    K_{s}(z,w) : = \frac{1 - 
    s(z) \overline{s(w)}}{1 - z \overline{w}} \succeq 0.
    \end{equation}
    
    \item $s$ has a unitary transfer-function realization, i.e., 
    there is a unitary colligation matrix $\bU = \left[ 
    \begin{smallmatrix} A & B \\ C & D \end{smallmatrix} \right] 
    \colon \cX \oplus {\mathbb C} \to \cX \oplus 
    {\mathbb C}$ so 
    that
    \begin{equation}   \label{Itransfunc}
	s(z) = D + z C (I - zA)^{-1} B.
    \end{equation}
    
    \item $s$ satisfies the von Neumann inequality: for any 
    strict contraction operator $T$ on a Hilbert space $\cK$, $\|s(T) 
    \| \le 1$.
\end{enumerate}
\end{theorem}

A natural multivariable generalization of the Schur class from this point of view 
is to consider functions $s$ defined on the polydisk ${\mathbb D}^{d}$ 
(where $d$ is a positive integer).  It has been known for some time 
that the von Neumann inequality fails in more than two variables, 
i.e.:  if $d>2$ there is a holomorphic function $s$  on ${\mathbb 
D}^{d}$ (even a polynomial) with $\|s\|_{{\mathbb D}^{d}} \le 1$ and 
a commuting $d$-tuple $T = (T_{1}, \dots, T_{d})$ of strict 
contraction operators on a  Hilbert space ${\mathcal K}$ for which 
the multivariable von Neumann inequality
\begin{equation}   \label{vN}
  \| s(T)\| \le \|s\|_{{\mathbb D}^{d}}
\end{equation}
fails.  Nevertheless, the subclass of those Schur-class functions 
over ${\mathbb D}^{d}$ for which \eqref{vN} does hold, now called the 
{\em Schur-Agler class}, does have characterizations analogous to 
those given in Theorem \ref{T:classicalSchur} for the single-variable 
case (see \cite{Agler-Hellinger, AMcC99, BT}).  Note that the 
analogue of condition (4) in Theorem \ref{T:classicalSchur} is now 
used as the definition of the Schur-Agler class. 
We then have the following analogue of Theorem \ref{T:classicalSchur}

\begin{theorem}  \label{T:classicalSA}
    
    Given $s \colon {\mathbb D}^{d} \to {\mathbb C}$, the following are 
    equivalent.
    \begin{enumerate}
	\item $s \in \mathcal{SA}_{d}$.
	
	\item There are positive kernels $K_{1}, \dots, K_{d}$ on 
	${\mathbb D}^{d}$ so that 
\begin{equation}   \label{scalarAD}
	    1 - s(z) \overline{s(w)} = \sum_{k=1}^{d} (1 - z_{k} 
	    \overline{w_{k}}) K_{k}(z,w).
\end{equation}

\item There is a unitary colligation matrix $\bU = \left[ 
\begin{smallmatrix} A & B \\ C & D \end{smallmatrix} \right] \colon 
\cX \oplus {\mathbb C} \to \cX \oplus {\mathbb C}$ and a collection 
$\{P_{1}, \dots P_{d} \}$ of orthogonal projections with $P_{i} P_{j} 
= 0$ for $i \ne j$ and with $\sum_{j=1}^{d} P_{j} = I_{\cX}$ so that 
\begin{equation}   \label{Aglertransfunc}
    s(z) = D + C (I - Z(z) A)^{-1} Z(z) B
\end{equation}
where we have set $Z(z) = z_{1} P_{1} + \cdots + z_{d} P_{d}$.
\end{enumerate}

 \end{theorem}

In the test-function approach to defining generalized Schur-Agler 
classes, going back to the unpublished preprint of 
Agler \cite{Agler-preprint}  and developed further in
\cite{AMcC-book, DMM07, DM07, McCS}, one proceeds as follows.  
We here describe the scalar-valued function setting, although the 
paper \cite{DMM07} deals with a more general semigroupoid setting.
One replaces the unit disk ${\mathbb D}$ (or unit polydisk ${\mathbb 
D}^{d}$) with a completely general point set $\Omega$ and supposes 
that one is given a collection of ${\mathbb C}$-valued functions $\Psi$
on $\Omega$ (the set of {\em test functions}) subject to the condition that $\sup_{\psi \in \Psi} 
|\psi(z)| < 1$ for each $z \in \Omega$.  The set $\Psi$ carries with it a 
natural completely regular topology, namely, 
the weakest topology with respect to which each of the functions
\begin{equation}   \label{bbE}
    {\mathbb E}(z) \colon \psi \to \psi(z), \quad z \in \Omega
 \end{equation}
 is continuous.  One then says that a positive kernel $k$ is 
 $\Psi$-admissible (written as $k \in \cK_{\Psi}$) if multiplication by $\psi$ is contractive as an 
 operator on the reproducing kernel Hilbert space $\cH(k)$ associated 
 with $k$, i.e., if the kernel $K_{\psi,k}(z,w) = (1 - \psi(z) 
 \overline{\psi(w)} k(z,w)$ is positive for each $\psi \in \Psi$.
 We then say that the function $s \colon \Omega \to {\mathbb C}$ is 
 in the $\Psi$-Schur-Agler class $\mathcal{SA}_{\Psi}$ if 
 multiplication by $s$ is contractive on $\cH(k)$ for each $k \in 
 \cK_{\Psi}$, i.e., if the kernel $K_{s,k}(z,w) = (1 - s(z) 
 \overline{s(w)}) k(z,w)$ is a positive kernel for each $k \in 
 \cK_{\Psi}$.  We mention that the choice 
 \begin{equation} \label{Psi-clasSchur}
     \Omega = {\mathbb D}, \quad  
 \Psi = \{ \psi_{0}(z) = z\}
 \end{equation}
 leads to the classical Schur class 
 while the choice 
 \begin{equation}  \label{Psi-clasSA}
     \Omega = {\mathbb D}^{d}, \quad \Psi = \{ 
 \psi_{k}(z) = z_{k} \colon k = 1, \dots, d\}
 \end{equation}
 (where $z = (z_{1}, \dots, z_{d}) \in {\mathbb D}^{d}$) leads to the classical 
 Schur-Agler class $\mathcal{SA}_{d}$.
 
 The following is the main result concerning the 
 Schur-Agler class $\mathcal{SA}_{\Psi}$ associated with a general
 test-function collection $\Psi$.  
 
 \begin{theorem}  \label{T:0} (See \cite{DMM07, DM07} and 
     \cite{Ambrozie04} for an early version.)
     Given a function $s \colon \Omega \to {\mathbb C}$, the 
     following are equivalent.
     \begin{enumerate}
     \item $s \in \mathcal{SA}_{\Psi}$.
     
     \item There is a measure $\nu$ on $\Psi_{\beta}$ (the 
     Stone-\v{C}ech compactification of $\Psi$) and a measurable 
     family $\{K_{\psi} \colon \psi \in \Psi_{\beta}\}$ of positive 
     kernels on $\Psi_{\beta}$ so that 
   \begin{equation}   \label{IscalarAD}
       1 - s(z) \overline{s(w)} = \int_{\Psi_{\beta}} \left(1 - \psi(z) 
       \overline{\psi(w)}\right) K_{\psi}(z,w)\, {\tt d}\nu(\psi).
   \end{equation}
   
   \item There is a $C(\Psi_{\beta})$-unitary colligation, i.e., a 
   bock unitary operator $\bU = \left[ \begin{smallmatrix} A & B  \\ 
   C & D \end{smallmatrix} \right] \colon {\mathcal X} \oplus 
   {\mathbb C} \to {\mathcal X} \oplus {\mathbb C}$ together with a 
   $*$-representation $\rho$ of the $C^{*}$-algebra $C(\Psi_{\beta})$ 
   (continuous complex-valued functions on $\Psi_{\beta}$) into 
   $\cL(\cX)$ (bounded linear operators on $\cX$),  so that 
   \begin{equation} \label{Itransfunc'}
       s(z) = D + C (I - \rho({\mathbb E}(z)) A)^{-1} \rho({\mathbb 
       E}(z)) B
   \end{equation}
   (where ${\mathbb E}(z)$ is as in \eqref{bbE}).
  \end{enumerate}   
 \end{theorem}
 Note that conditions (2) and (3) in Theorem \ref{T:0} become 
 conditions (2) and (3) in Theorem \ref{T:classicalSchur} when 
 $\Omega$ and $\Psi$ are chosen as in \eqref{Psi-clasSchur},
 and conditions (2) and (3) in Theorem \ref{T:classicalSA} when 
 $\Omega$ and $\Psi$ are chosen as in \eqref{Psi-clasSA}.

 A different type of extension of the classical Schur class over the 
 unit disk  is the Schur-class 
 $\cS_{\cR}$ over a bounded, finitely connected planar domain ${\mathcal R}$.
 Here $\cR$ is a bounded domain in the complex plane with boundary consisting 
 of $m+1$ disjoint smooth Jordan curves $\partial_{0}, \partial_{1}, \dots, 
 \partial_{m}$, where $\partial_{0}$ denotes the boundary of the 
 unbounded component of the complement of $\cR$, and we define $\cS_{\cR}$ as the 
 class of all holomorphic functions from $\cR$ into the closed disk 
 ${\mathbb D}^{-}$.  Work in \cite{DMM07, DM07} identifies the Schur 
 class $\cS_{\cR}$ over $\cR$ as a test-function Schur-Agler  class 
 $\mathcal{SA}_{\Psi_{\cR}}$ for a certain collection of test functions 
 $\Psi_{\cR} = \{ \psi_{\bx} \colon \bx \in {\mathbb T}_{\cR\}}$ indexed 
 by the so-called $\cR$-torus ${\mathbb T}_{\cR}$ defined as the 
 Cartesian product of the connected components of $\partial \cR$:
 $$
  \bx \in   {\mathbb T}_{\cR} : = \partial_{0} \times \partial_{1} \times 
   \cdots \times \partial_{m}.
 $$
(see Section \ref{S:cR} below for complete details).  In particular, 
the decomposition \eqref{IscalarAD} in Theorem \ref{T:0} for this case gives 
us the following: {\em  given $s \in \cS_{\cR}$, there is a measure $\nu$ 
on ${\mathbb T}_{\cR}$ and a family of positive kernels $\{k_{\bx} 
\colon \bx \in {\mathbb T}_{\cR} \}$ so that} 
\begin{equation}  \label{scalarRAdecom}
    1 - s(z) \overline{s(w)}  = \int_{{\mathbb T}_{\cR}} \left(1 - 
    \psi_{\bx}(z) \overline{\psi_{\bx}(w)}\right) k_{\bx}(z,w)\, {\tt 
    d}\nu(\bx).
\end{equation}

We shall be interested in matrix- and operator-valued versions of 
these Schur and Schur-Agler classes.  
The operator-valued version of the Schur class over $\cR$, which we 
denote as $\cS_{\cR}(\cU, \cY)$, consists of holomorphic functions $S$ on 
 $\cR$ with values $S(z)$ equal to contraction operators 
between two Hilbert spaces $\cU$ and $\cY$. For the case where $\cR = 
{\mathbb D}$, we drop the subscript $\cR$ and write simply $\cS(\cU, 
\cY)$;  we also abbreviate $\cS_{\cR}(\cU, \cU)$ to $\cS_{\cR}(\cU)$. 
There is also an operator-valued version of the Schur-Agler class 
over ${\mathbb D}^{d}$, namely:  {\em $S \colon {\mathbb D}^{d} \to \cL(\cU, 
\cY)$ is in the Schur-Agler class $\mathcal{SA}_{d}(\cU, \cY)$ if $S$ 
is a holomorphic map from ${\mathbb D}^{d}$ into $\cL(\cU, \cY)$ such 
that $\| S(T) \| \le 1$ for any commutative tuple $T = (T_{1}, \dots, 
T_{d})$ of strictly contractive operators on a Hilbert space $\cK$}, 
where we use a tensor functional calculus to define $S(T)$:
$$
S(T) = \sum_{n \in {\mathbb Z}^{d}_{+}} S_{n} \otimes T^{n} \text{ if 
}  S(z) = \sum_{n \in {\mathbb Z}^{d}} S_{n} z^{n}
$$
where we use standard multivariable notation:
$$
 z^{n} = z_{1}^{n_{1}} \cdots z_{d}^{n_{d}}, \quad
 T^{n} = T_{1}^{n_{1}} \cdots T_{d}^{n_{d}} \text{ for } n = (n_{1}, 
 \dots, n_{d}) \in {\mathbb Z}^{d}_{+}.
$$
Then Theorems \ref{T:classicalSchur} and  \ref{T:classicalSA} have 
seamless extensions to the matrix-/operator-valued settings.  Indeed, 
$S \in \cS(\cU, \cY)$ if and only if the de Branges-Rovnyak 
$\cL(\cY)$-valued  kernel 
$$
  K_{S}(z,w): = \frac{ I_{\cY} - S(z) S(w)^{*}}{ 1 - z \overline{w}}
$$
is a positive kernel on ${\mathbb D}$ if and only if there is a 
unitary colligation matrix $\bU = \left[ \begin{smallmatrix} A & B \\ 
C & D \end{smallmatrix} \right] \colon \cX \oplus \cU \to \cX \oplus 
\cY$ so that $S(z) = D + z C ( I - z A)^{-1} B$.  Similarly, $S \in 
\mathcal{SA}_{d}(\cU, \cY)$ if and only if there are positive 
$\cL(\cE)$-valued kernels $K_{1}, \dots, K_{d}$ on ${\mathbb D}^{d}$ 
so that $I - S(z) S(w)^{*} = \sum_{k=1}^{d} (1 - z_{k} 
\overline{w_{k}}) K_{k}(z,w)$ if and only if $S$ has a representation 
as in \eqref{Aglertransfunc} but with $\bU$ acting from $\cX \oplus 
\cU$ to $\cX \oplus \cY$. We mention that this
result has inspired several variants where the polydisk 
${\mathbb D}^{d}$ is replaced by a more general domain $\cD_{Q}$ in 
${\mathbb C}^{d}$ specified by a polynomial (or more generally 
analytic) matrix-valued determining function $Q$:  $\cD_{Q} = \{ z 
\in {\mathbb C}^{d} \colon \| Q(z) \| < 1\}$; more generally the 
technique of the proof going through the transfer-function 
realization naturally leads to interpolation and  commutant lifting 
versions of the result (see \cite{BT, BLTT, Tomerlin, BTV, AT, BBQ, 
AE}).  We mention  that  there is now also a noncommutative version of the 
Schur-Agler class \cite{BGM}.

However, for the case $\cS_{\cR}({\mathbb 
C}^{N})$, the expected matrix generalization of 
\eqref{scalarRAdecom}, namely
\begin{equation}   \label{expected}
    I - S(z) S(w)^{*} = \int_{{\mathbb T}_{\cR}} \left( 1 - 
    \psi_{\bx}(z) \overline{\psi_{\bx}(w)} \right) K_{\bx}(z,w) \, 
    {\tt d}\nu(\bx)
\end{equation}
for a measurable family $\{ K_{\bx} \colon \bx \in {\mathbb 
T}_{\cR}\}$ of positive $N \times N$ matrix-valued kernels on $\cR$,
fails in general, at least in the case where $\cR$ is a region with 
three holes having some additional symmetry properties; indeed this 
phenomenon is a key ingredient in the negative answer to the spectral 
set question for such regions $\cR$ obtained by Dritschel and 
McCullough in \cite{DM05}.  

One of the main motivations for the 
present paper is to develop a framework of test-function Schur-Agler 
class $\mathcal{SA}_{\Psi}$ for the case of matrix- or 
operator-valued test functions $\Psi$ and to recover a formula of the 
type \eqref{expected} for the Schur class $\cS_{\cR}({\mathbb 
C}^{N})$ for an appropriately enlarged class $\Psi^{N}_{\cR}$ of 
matrix-valued test functions.  We therefore develop a
systematic extension of the work of \cite{DMM07, DM07} to the matrix- 
and operator-valued setting: this is the main 
content of Section \ref{S:main} below.   We also emphasize the 
interpolation version of the main result, whereby one characterizes 
which functions $S_{0}$ defined on some subset $\Omega_{0}$ of 
$\Omega$ can be extended to a test-function Schur-Agler-class function $S$
defined on all of $\Omega$.  Most of the analysis builds 
on the earlier work of \cite{Agler-Hellinger, AMcC99, BT, AT, BBQ, 
Ambrozie04, DMM07, DM07}, but there are places where new 
ideas and techniques were required.

In Section \ref{S:ex} we take
two algebras which 
are intrinsically defined and identify their unit balls as also arising as 
test-function Schur-Agler classes. The first has already been 
mentioned: namely, the algebra of bounded holomorphic $N \times N$ 
matrix functions over a multiply-connected planar domain $\cR$ whose 
unit ball is the Schur class $\cS_{\cR}({\mathbb C}^{N})$.  The 
second is the matrix-valued version of the constrained Hardy 
algebra over the unit disk ${\mathbb D}$ (bounded holomorphic 
functions $f$ on ${\mathbb D}$ subject to the constraint that $f'(0) = 
0$). The first example has been an object of much study over the 
years (see \cite{Abrahamse, Ball79, BC, AHR, DM05, VF}) while interest 
in the second is more recent \cite{DPRS, BBtH, Rag08}.
Motivation for study of the second algebra comes from the fact that 
it is a model for the bounded analytic functions on the intersection of a variety $V$ 
embedded in ${\mathbb C}^{2}$ with the unit bidisk (see \cite{AMcC05}).
For these two examples we identify an appropriate class of 
test functions $\Psi^{N}$ so that the unit ball of the given algebra is equal to the 
matrix-valued test-function Schur-Agler class 
$\mathcal{SA}_{\Psi^{N}}$  associated with $\Psi^{N}$.
It is always possible to choose $\Psi^{N}$ simply as the unit ball of the 
given algebra; the point is to find a valid class $\Psi^{N}$ which is as 
small as possible.  As has already been mentioned for the first 
example,  in both examples the test-function class $\Psi^{1}$ 
identified in previous work (\cite{DM07, DP}) 
 for the scalar-valued version fails to work for 
the matrix-valued case.  For each of these two examples, we find a 
valid test-function class $\Psi^{N}$ as a linear-fractional transform of the set 
of extreme points of a normalized matrix-valued  Herglotz (positive real part) 
version of the algebra, just as has been done for the scalar-valued 
case in \cite{DM05, DM05, DP}.  Identification of these extreme 
points for the matrix-valued case leads us to draw on results from 
\cite{BG} concerning extreme points for a convex cone of matrix 
quantum probability measures (positive matrix-valued measures with 
total mass equal to the identity matrix).  The resulting 
test-function classes are not as explicit as in the scalar-valued 
settings; however, for the Schur class $\cS_{\cR}$ with $\cR$ equal 
to an annulus, we are able to use results of McCullough \cite{McC95} to 
obtain a more explicit test-function class and use the resulting 
matrix-valued continuous Agler decomposition (the matrix-valued 
analogue of \eqref{IscalarAD}) to obtain a variant of McCullough's 
positive solution of the spectral set question for an annulus.

A criticism of the study of Schur-Agler classes in general is that their 
intrinsic structure is a priori mysterious: after going through the 
several steps of the definition, one does not have any intrinsic 
characterization of the eventual result.  Our work in Section 
\ref{S:ex} (as well as the work in \cite{DM07, DP}) counterbalances 
this concern by starting with an intrinsically defined function 
algebra and identifying it as a Schur-Agler class.  There are now papers 
obtaining characterizations of which operator algebras have unit 
balls equal to a Schur-Agler class (see \cite{MP, Jury}).  Other work 
\cite{JKM} characterizes families of kernels so that the 
associated contractive multipliers form a test-function Schur-Agler 
class. It should be of interest to extend these results to 
the matrix-valued setting in the spirit of the present paper.

The paper is organized as follows.  Section \ref{S:prelim} presents 
some preliminary material on test functions, positive kernels, and 
structured unitary colligation matrices needed in the sequel.  Section 
\ref{S:main} presents the main structure result (including the 
interpolation version as well as a representation-theoretic version) 
for the general matrix-valued test-function Schur-Agler class.  
Section \ref{S:ex} develops the two illustrative examples of 
matrix-valued Schur classes which can be identified as test-function 
Schur-Agler classes.  Finally we mention that  this paper together 
with \cite{BG} form an enhanced version of the second author's 
dissertation \cite{GH}.

\section{Preliminaries}  \label{S:prelim}

\subsection{Test functions}  \label{S:test}

We assume that we are given two coefficient Hilbert spaces $\cU_{T}$ 
and $\cY_{T}$ and a collection $\Psi$ of functions $\psi$ on the 
abstract set of points $\Omega$ with values in the space $\cL(\cU_{T}, \cY_{T})$ of 
bounded linear operators between $\cU_{T}$ and $\cY_{T}$.  We say 
that $\Psi$ is a {\em collection of test functions} if it happens that
\begin{equation}  \label{test-axiom}
    \sup \{ \| \psi(z) \| \colon \psi \in \Psi \} < 1 \text{ for each } 
    z \in \Omega.
\end{equation}
We view $\Psi$ as a subset of $B(\Omega, \overline{\cB}\cL(\cU_{T}, 
\cY_{T}))$ (the space of (bounded) maps from $\Omega$ into the closed 
unit ball of bounded linear operators between $\cU_{T}$ and 
$\cY_{T}$).   We topologize 
$B(\Omega, \overline{\cB}\cL(\cU_{T}, \cY_{T}))$ with the topology of 
pointwise weak-$*$ convergence, i.e., we view $B(\Omega, \overline{\cB}\cL(\cU_{T}, \cY_{T}))$
as the Cartesian product $\Pi_{\Omega} \overline{\cB}\cL(\cU_{T}, 
\cY_{T})$ with the standard Cartesian product topology).  As such 
$B(\Omega, \overline{\cB}\cL(\cU_{T}, \cY_{T}))$ is compact by 
Tychonoff's Theorem (\cite[Theorem XI.1.4]{Dugundji}),  since
each fiber $\overline{\cB}\cL(\cU_{T}, \cY_{T})$ is compact by the 
Banach-Alaoglu Theorem \cite[Theorem 3.15]{Rudin}.   As a subspace of the completely regular 
space  $B(\Omega, \overline{\cB}\cL(\cU_{T}, \cY_{T}))$ (i.e., 
$B(\Omega, \overline{\cB}\cL(\cU_{T}, \cY_{T}))$ is Hausdorff and 
any closed set can be separated from a point disjoint from it by a 
continuous function), $\Psi$ is completely regular in the subspace 
topology inherited from 
$B(\Omega, \overline{\cB}\cL(\cU_{T}, \cY_{T}))$.
The closure of $\Psi$ in this topology is compact; however 
we shall be more interested in the Stone-\v{C}ech compactification 
$\Psi_{\beta}$ of $\Psi$ \cite[Section XI.8]{Dugundji}.
Then the space $C_{b}(\Psi, \cL(\cH, \cK))$ of bounded continuous 
functions $f$ from $\Psi$ into a space $\cL(\cH, \cK)$ of 
bounded linear operators between two Hilbert spaces $\cH$ and $\cK$
can be identified with  the space $C(\Psi_{\beta}, \cL(\cH, \cK))$ 
of continuous functions from the Stone-\v{C}ech compactification 
$\Psi_{\beta}$ into $\cL(\cH, \cK)$.  An operator-valued version 
of the Riesz representation theorem 
allows us to identify the dual of 
$C_{b}(\Psi, \cL(\cH, \cK))$ with regular, bounded, weakly countably 
additive  $\cC_{1}(\cK, \cH)$-valued measures on $\Psi_{\beta}$, 
where we use the notation $\cC_{1}(\cK, \cH)$ to denote the 
trace-class operators from $\cK$ to $\cH$.  We note that 
there are continuous linear functionals $L$  in  $C(\Psi_{\beta}, \cL(\cH, \cK))$ 
such that allowing points of $\Psi_{\beta} \setminus \Psi$ to be part of the support of 
the corresponding measure $\mu_{L}$ is essential (see 
\cite[Section 5.2]{DM07}).

For each $\psi \in \Psi$ we 
define the map $\ev_{\psi} \colon C_{b}(\Psi, \cL(\cH, \cK)) \to 
\cL(\cK)$ by $\ev_{\psi} \colon f \to f(\psi)$.  A particular element 
of $C_{b}(\Psi, \cL(\cU_{T}, \cY_{T}))$ which will often come up is 
the function ${\mathbb E}(z)$ (for each $z \in \Omega$) given by
\begin{equation}  \label{boldE}
  \ev_{\psi}({\mathbb E}(z)) = {\mathbb E}(z)(\psi): = \psi(z).
\end{equation}

\subsection{Positive operator-valued kernels and their multipliers}  
\label{S:kernels}

Let $\cE$ be any Hilbert space and suppose that $K$ is a function on 
$\Omega \times \Omega$ with values in $\cL(\cE)$.  We say that  
$K$ is a {\em positive kernel} if the Aronszajn condition
\begin{equation}   \label{posker}
 \sum_{i,j = 1}^{N} \langle K(z_{i}, z_{j}) e_{j}, e_{i} \rangle_{\cE} 
 \ge 0 \text{ for all } z_{1}, \dots, z_{n} \in \Omega,\,
 e_{1}, \dots, e_{N} \in \cE, \, N = 1, 2, \dots.
 \end{equation}
The following equivalent versions of the 
positive-kernel condition are often used in function-theoretic 
operator theory settings.

\begin{theorem} \label{T:posker}  (See e.g.~\cite{AMcC-book}.)
Suppose that we are give a function $K \colon \Omega \times \Omega \to \cL(\cE)$. 
Then the following are equivalent:

\begin{enumerate}
    \item $K$ is a positive kernel, i.e., condition \eqref{posker} 
    holds.
    
    \item There is a Hilbert space $\cH(K)$ consisting of $\cE$-valued functions $f$ 
    such that $K(\cdot, w) e \in \cH(K)$ for each $w \in \Omega$ and $e \in \cE$ and
has the reproducing property:
$$
\langle f, K(\cdot, w) e \rangle_{\cH(K)} = \langle f(w), e 
\rangle_{\cE} \text{ for all } f \in \cH(K).
$$

\item $K$ has a Kolmogorov decomposition:  there is an auxiliary 
Hilbert space $\cX$ and a function $H \colon \cX \to \cE$ so that
\begin{equation}   \label{Kolmogorov}
    K(z,w) = H(z) H(w)^{*}.
\end{equation}
In fact one can take $\cX$ to be the reproducing kernel Hilbert space 
$\cH(K)$ described in (2) above with $H(z) = \ev_{z} \colon f \mapsto 
f(z)$.
\end{enumerate}
\end{theorem}

Rather than using a positive kernel to construct a reproducing kernel 
Hilbert space as in condition (2) in Theorem \ref{T:posker}, it is 
also possible to construct a reproducing kernel Hilbert module as 
follows.  By a Hilbert module over a $C^{*}$-algebra $\mathfrak B$ we 
mean a linear space $E$ which is a right module over ${\mathfrak B}$ 
which is also equipped with an ${\mathfrak B}$-valued inner product
and satisfies additional compatibility requirements with respect to 
the algebra structure of ${\mathfrak B}$ (see \cite[Section 2.1]{RW}):
$$
\langle \cdot, \cdot \rangle_{E} \colon E \times E \to {\mathfrak B}
$$
which satisfies the usual inner product axioms:
\begin{enumerate}
    \item $\langle \lambda x + \mu y, z \rangle = \lambda  \langle x, 
    z \rangle + \mu \langle y, z \rangle$, 
    \item $\langle x\cdot b, y \rangle = \langle x, y \rangle  b$,
    \item $\langle x, y \rangle^{*} = \langle y, x \rangle$,
    \item $\langle x, x \rangle \ge 0$ (as an element of ${\mathfrak B}$),
    \item $\langle x, x \rangle = 0$ implies that $x = 0$,
    \item $E$ is complete in the norm given by $\|x \| = \| \langle 
    x, x \rangle \|^{1/2}_{{\mathfrak A}}$
\end{enumerate}
for all $x,y,z \in E$, $b \in {\mathfrak B}$ and $\lambda, \mu \in 
{\mathbb C}$.  (Here we follow the mathematicians'(rather than the 
physicists') convention that 
inner products are linear in the left slot; this departs from the 
standard usage in the operator-algebra literature.)
By modifying the construction of $\cH(K)$ in Theorem \ref{T:posker}, 
one can construct a $C^{*}$-module, denoted as $\bcH(K)$, over the 
$C^{*}$-algebra $\cL(\cE)$ characterized as follows.

\begin{theorem}  \label{T:RKC*mod}
    Suppose that $K \colon \Omega \times \Omega \to \cL(\cE)$ is a 
    positive kernel as in \eqref{posker}.  Then there is a uniquely 
    determined $C^{*}$-module $\bcH(K)$ over ${\mathfrak B} = \cL(\cE)$ 
    with the following properties:
    \begin{enumerate}
	\item $\bcH(K)$ consists of $\cL(\cE)$-valued functions on 
	$\Omega$,
	\item for each $w \in \Omega$, $K(\cdot, w)$ is in $\bcH(K)$ 
	and the span of such elements is dense in $\bcH(K)$, and
	\item for each $F \in \bcH(K)$,
	$$
	\langle F, K(\cdot, w) \rangle_{\bcH(K)} = F(w) \in \cL(\cE).
	$$
	\end{enumerate}
\end{theorem}

\begin{proof}  Define an inner product on a pair of kernel elements 
    $K(\cdot, w)$ and $K(\cdot, z)$ by
    $$
    \langle K(\cdot, w), K(\cdot, z) \rangle_{\bcH(K)} = K(z,w)
    $$
    and extend by linearity to the space of kernel elements.  Mod out 
    by any linear combinations having zero self inner product and 
    take the completion to arrive at the space $\bcH(K)$ having all 
    the asserted properties.  Note that there is a version of the 
    Cauchy-Schwarz inequality available (see \cite[Lemma 2.5]{RW}) 
    which guarantees that the point evaluation map $\ev \colon f 
    \mapsto f(w)$ extends to elements of the completion, and hence 
    elements of the completion can also be identified as 
    $\cL(\cE)$-valued functions on $\Omega$.
\end{proof}

It is natural now to take the next step and introduce the notion of 
$C^{*}$-corres\-pondence (see \cite{MS98}). Given two $C^{*}$-algebras 
${\mathfrak A}$ and ${\mathfrak B}$, by an $({\mathfrak A}, 
{\mathfrak B})$-correspondence we mean a Hilbert module $E$ over 
${\mathfrak B}$ which also carries a left ${\mathfrak A}$-action $x 
\mapsto a \cdot x$ which is a $*$-representation of ${\mathfrak A}$ 
with respect to the ${\mathfrak B}$-valued inner product on $E$:
$$
\langle a \cdot x, y \rangle_{E} = \langle x, a^{*} \cdot y 
\rangle_{E}.
$$

Given three 
$C^{*}$-algebras ${\mathfrak A}$, ${\mathfrak B}$ and ${\mathfrak C}$
together with an $({\mathfrak A}, {\mathfrak B}$)-correspondence $E$ 
and a $({\mathfrak B}, {\mathfrak C})$-correspondence $F$, the 
internal tensor product $E \otimes F$ of $E$ and $F$ is defined to be 
the $({\mathfrak A}, {\mathfrak C})$-correspondence generated as the 
Hausdorff completion of the  span of pure 
tensors $e \otimes f$ ($e \in E$ and $f \in F$) in the ${\mathfrak 
C}$-valued inner product given by
\begin{equation}  \label{tensor1}
\langle e \otimes f,\,  e' \otimes f' \rangle_{E \otimes F} =
\langle \left( \langle e, e' \rangle_{E} \right) \cdot f, \, f' 
\rangle_{F}
\end{equation}
with left ${\mathfrak A}$-action given by
\begin{equation}   \label{tensor2}
  a \cdot (e \otimes f) = (a \cdot e) \otimes f.
\end{equation}
It is routine to verify that one then gets the balancing property
\begin{equation}   \label{balance}  
e \otimes (b \cdot f) = (e \cdot b) \otimes f
\end{equation}
for $e \in E$, $f \in F$ and $b \in {\mathfrak B}$.  

We shall need a couple of applications of this internal 
tensor-product construction.  The first is as follows. 
For $K$ an $\cL(\cE)$-valued positive kernel on $\Omega$, we view the 
$C^{*}$-module over ${\mathfrak B}$ constructed in Theorem 
\ref{T:RKC*mod} as a $({\mathbb C}, \cL(\cE))$-correspondence.
For $\cX$ another coefficient Hilbert space, let $\cC_{2}(\cX, \cE)$ 
be the space of Hilbert-Schmidt class operators from $\cX$ into $\cE$.
Then $\cC_{2}(\cX, \cE)$ has a standard Hilbert-space inner product
$$
\langle T, T' \rangle_{\cC_{2}(\cX, \cE)} = \operatorname{tr} (T 
T^{\prime *}).
$$
We also have a left action of the $C^{*}$-algebra $\cL(\cE)$ on 
$\cC_{2}(\cX, \cE)$ via left multiplication:
$$
  X \cdot T = X T \text{ for } X \in \cL(\cE),\, T \in \cC_{2}(\cX, 
  \cE)
$$
and this action gives rise to a $*$-representation of $\cL(\cE)$ on 
$\cC_{2}(\cX, \cE)$:
\begin{align*}
  \langle X \cdot T, \, T' \rangle_{\cC_{2}(\cX, \cE)} &  = 
  \langle X  T, T' \rangle_{\cC_{2}(\cX, \cE)} = \operatorname{tr} 
  (XTT^{\prime *}) =  \operatorname{tr}(T T^{\prime *} X)    \\  & =
  \operatorname{tr}(T (X^{*} T')^{*}) =
  \langle T,\,  X^{*} \cdot T' \rangle_{\cC_{2}(\cX, \cE)}.
\end{align*}
In this way we may view $\cC_{2}(\cX, \cE)$ as an $(\cL(\cE), 
{\mathbb C})$-correspondence.  We may then form the internal 
$C^{*}$-correspondence tensor-product $\bcH(K) \otimes \cC_{2}(\cX, 
\cE)$. Explicitly, the inner product on pure tensors $F \otimes T$
($F \in \bcH(K)$, $T \in \cC_{2}(\cX, \cE)$ is given by
$$ 
\langle F \otimes T,\,  F' \otimes T' \rangle_{\bcH(K) \otimes 
\cC_{2}(\cX, \cE)} = \operatorname{tr} \left( \langle F, F' 
\rangle_{\bcH(K)} T T^{\prime *} \right).
$$

When we evaluate the first factor $F$ in a pure tensor $F \otimes T$ 
at  a point $w$ in $\Omega$, we get a tensor of the form 
$$
F(w) \otimes T  \in \cL(\cE) \otimes \cC_{2}(\cX, \cE) \cong 
\cC_{2}(\cX, \cE).
$$
To interpret this tensor product as a $C^{*}$-correspondence internal 
tensor product, we view  $\cL(\cE)$  as a $(\cL(\cE), 
\cL(\cE))$-correspondence with inner product
$\langle X, X' \rangle = X^{\prime *} X \in \cL(\cE)$ and left 
action given by left multiplication:
$X' \cdot X = X'X$.
The balancing property \eqref{balance} then leads to the 
identification $\cL(\cE) \otimes \cC_{2}(\cX, \cE) \cong \cC_{2}(\cX, 
\cE)$. 

Using a linearity and approximation argument, one can show that in 
fact elements $H$ of $\bcH(K) \otimes \cC_{2}(\cX, \cE)$ can be 
viewed as $\cC_{2}(\cX, \cE)$-valued functions on $\Omega$ such that 
$K(\cdot, w) U \in \bcH(K) \otimes \cC_{2}(\cX, \cE)$ for each $w \in 
\Omega$ and $U \in \cC_{2}(\cX, \cE)$, and the kernel element 
$K(\cdot, w)U$ has the reproducing property
$$
\langle G, K(\cdot, w) U \rangle_{\bcH(K) \otimes \cC_{2}(\cX, \cE)} 
= 
\langle G(w) , U \rangle_{\cC_{2}(\cX, \cE)} : = \operatorname{tr} 
\left( G(w) U^{*} \right).
$$
Thus $\bcH(K) \otimes \cC_{2}(\cX, \cU)$ is a reproducing kernel 
Hilbert space in the sense of Theorem \ref{T:posker} when we identify 
the range space $\cL(\cE)$ of $K$ as the subspace of $\cL(\cC_{2}(\cX, 
\cE))$ consisting of left multiplication operators by elements of 
$\cL(\cE)$:
$$
 X \in \cL(\cE) \mapsto L_{X} \in \cL(\cC_{2}(\cX, \cE)): \,
 L_{X} \colon T \mapsto X T
 $$
 and we view $\cC_{2}(\cX, \cE)$ as a Hilbert space in the inner 
 product
 $$
 \langle T, T' \rangle_{\C_{2}(\cX, \cE)} : = \operatorname{tr} 
 \left( T T^{\prime *} \right).
 $$
 In the sequel it will be convenient to use the shorthand notation
 \begin{equation}   \label{shorthand}
     \cH(K)_{\cX}: = \bcH(K) \otimes \cC_{2}(\cX, \cE).
 \end{equation}
 Note that in this notation, if $\cH(K)$ is as in Theorem 
 \ref{T:posker}, then we have $\cH(K) = \cH(K)_{{\mathbb C}}$.
 
 \begin{remark} \label{R:shorthand}
     {\em The space $\cH(K)_{\cX}$ could just as well have been 
     constructed as equal to the space $\cH(K) \otimes \cC_{2}(\cX, 
     {\mathbb C})$ where the spaces  $\cH(K)$ (defined as in Theorem 
     \ref{T:posker}) and $\cC_{2}(\cX, {\mathbb C})$ (the dual 
     version of the Hilbert space $\cX$) are viewed as $({\mathbb C}, 
     {\mathbb C})$-correspondences (i.e., as ordinary Hilbert 
     spaces), and the tensor product reduces to the standard 
     Hilbert-space tensor product. 
    } \end{remark}
 
 Suppose that we are given two coefficient Hilbert spaces $\cU$ and 
 $\cY$ and an $\cL(\cU, \cY)$-valued function $S$ on $\Omega$.  We 
 define the right multiplication operator $R_{S}$ by
 $$
 \left( R_{S}(F) \right)(z) = F(z) S(z).
 $$
 Thus $R_{S}$ maps $\cC_{2}(\cU, \cE)$-valued functions on $\Omega$ 
 to $\cC_{2}(\cU, \cE)$-valued functions on $\Omega$.  Given a 
 positive $\cL(\cE)$-valued kernel $K$ on $\Omega$, it is of 
 interest to determine exactly when $R_{S}$ maps $\cH(K)_{\cY}$ 
 boundedly (or contractively) into $\cH(K)_{\cU}$.  The answer is 
 given by the following theorem.
 
 \begin{theorem}  \label{T:RSbounded} Let $K$ be an $\cL(\cE)$-valued 
     positive kernel on $\Omega$ and $S$ an $\cL(\cU, \cY)$-valued 
     function on $\Omega$.  Then the right multiplication operator 
     $R_{S}$ is bounded as an operator from $\cH(K)_{\cY}$ to 
     $\cH(K)_{\cU}$  with $\| R_{S} \| \le M$ if and only if the 
     ${\mathbb C}$-valued kernel
     \begin{equation}  \label{kXSKM}
     k_{X,S,K,M}(z,w): = \operatorname{tr}\left[ X(w)^{*}( M^{2} 
     I_{\cU} - S(w)^{*} S(z)) X(z) K(z,w) \right]
     \end{equation}
     is a positive kernel on $\Omega$ for each choice of function $X 
     \colon \Omega \to \cC_{2}(\cE, \cU)$.
 \end{theorem}
 
 \begin{proof}
    By rescaling it suffices to consider the case $M=1$ and $\| R_{S} 
    \| \le 1$.
    
    The computation
\begin{align*}
    \langle R_{S}f, K(\cdot, w) U \rangle_{\cH(K)_{\cU}} & =
    \langle f(w) S(w), U \rangle_{\cC^{2}(\cU, \cE)}  \\
    & = \operatorname{tr} \left(f(w) S(w) U^{*} \right) \\
    & = \operatorname{tr} \left(f(w) (U S(w)^{*}) \right)  \\
    & = \langle f, K(\cdot, w) U S(w)^{*} \rangle_{\cH(K)_{\cY}}
\end{align*}
shows that
$$
  (R_{S})^{*} \colon K(\cdot, w) U \mapsto K(\cdot, w) U S(w)^{*}
$$
whenever $R_{S}$ is well defined as an element of $\cL(\cH(K)_{\cY}, 
\cH(K)_{\cU})$.  As elements of the form $\sum_{j=1}^{N} K(\cdot, 
z_{j}) U_{j}$ are dense in $\cH(K)_{\cU}$, we see that $\| R_{S} \| 
\le 1$ holds if and only if
\begin{align*}
    0 \le & \left\| \sum_{j=1}^{N} K(\cdot, z_{j}) U_{j} 
    \right\|^{2} - \left\| 
    R_{S}^{*} \left( \sum_{j=1}^{N} K(\cdot, z_{j}) U_{j} \right) 
    \right\|^{2}  \\
    = &  \left\| \sum_{j=1}^{N} K(\cdot, z_{j}) U_{j} \right\|^{2} -
    \left\| \sum_{j=1}^{N} K(\cdot, z_{j}) U_{j} S(z_{j})^{*} 
    \right\|^{2}
 \end{align*}
 holds for all choices of $z_{1}, \dots, z_{N} \in \Omega$ and 
 $U_{1}, \dots, U_{N} \in \cC_{2}(\cU, \cE)$ and $N = 1,2, \dots$.
 Expanding out self inner products and using the invariance of the 
 trace under cyclic permutations converts this condition to
 \begin{align*}
     0 & \le \sum_{i,j=1}^{N} \operatorname{tr} \left( K(z_{i}, z_{j}) 
     U_{j} U_{i}^{*} - K(z_{i}, z_{j}) U_{j} S(z_{j})^{*} S(z_{i}) 
     U_{i}^{*} \right)   \\
     & = \sum_{i,j = 1}^{N} \operatorname{tr} \left( U_{j}(I - 
     S(z_{j})^{*} S(z_{i})) U_{i}^{*} K(z_{i} ,z_{j}) \right)  \\
     & = \sum_{i,j=1}^{N} \operatorname{tr} \left( X(z_{j})^{*} (I - 
     S(z_{j})^{*} S(z_{i})) X(z_{i}) K(z_{i}, z_{j}) \right)
 \end{align*}
 where we have set $X(z_{i}) = U_{i}^{*}$.  This positivity condition 
 holding for all choices of $z_{1}, \dots, z_{N} \in \Omega$ and 
 $X(z_{1}), \dots, X(z_{N}) \in \cC_{2}(\cE, \cU)$ for all $N = 1, 2, 
 \dots$ in turn is equivalent to the positivity of the kernel 
 $k_{X,S,K,1}$ on $\Omega$ for all choices of $X \colon \Omega \to 
 \cC_{2}(\cE, \cU)$.
  \end{proof}

  We shall also need a characterization of functional Hilbert spaces 
 of the form $\cH(K)_{\cX}$.
 
 \begin{theorem}  \label{T:HKXconverse}
     Suppose that $\cH$ is a Hilbert space whose elements are 
     $\cC_{2}(\cX, \cE)$-valued functions on $\Omega$.  Then there is 
     an $\cL(\cE)$-valued positive kernel $K$ on $\Omega$ such that 
     $\cH$ is isometrically equal to $\cH(K)_{\cX}$ if and only if
     \begin{enumerate}
	 \item the point evaluation map $\ev_{w} \colon f \mapsto 
	 f(w)$ defines a bounded operator from $\cH$ into 
	 $\cC_{2}(\cX, \cE)$ fo each $w \in \Omega$, and
	 \item $\cH$ is a right module over $\cL(\cX)$ with the right 
	 action of $\cL(\cX)$ commuting with each point evaluation 
	 map $\ev_{w}$:
	 \begin{equation}   \label{rightintertwine}
	 \ev_{w}(f \cdot X) = (\ev_{w} f) X \text{  or } (f \cdot 
	 X)(w) = f(w) X \text{ for all } w \in \Omega.
	 \end{equation}
 \end{enumerate}
  \end{theorem}
  
  \begin{proof}
      By Theorem \ref{T:posker}, from the fact that the point 
      evaluations $\ev_{w}$  are 
      bounded, we get that $\cH = \cH({\mathbf K})$ for an 
      $\cL(\cC_{2}(\cX, \cE))$-valued positive kernel ${\mathbf 
      K}(z,w) = \ev_{z} \cdot (\ev_{w})^{*}$.  The additional 
      condition \eqref{rightintertwine} then implies that ${\mathbf 
      K}(z,w)$ commutes with the right multiplication operators 
      $R_{X} \colon T \mapsto TX$ on $\cC_{2}(\cX, \cE)$ ($X \in 
      \cL(\cX)$).  This is enough to force ${\mathbf K}(z,w)$ to be a 
      left multiplication operator ${\mathbf K}(z,w) = L_{K(z,w)}$ for 
      a $K(z,w) \in \cL(\cE)$.  One next verifies that $K$ so 
      constructed is an $\cL(\cE)$-valued positive kernel and that we 
      recover $\cH$ as $\cH = \cH(K)_{\cX}$.
   \end{proof}
   
   We shall also have use for a far-reaching generalization of the
   positive kernels discussed so far introduced by Barreto, Bhat, 
   Liebscher, and Skeide in \cite{BBLS}.  Given two $C^{*}$-algebras 
   ${\mathfrak A}$ and ${\mathfrak B}$, we say that a function 
   $\Gamma$ on $\Omega \times \Omega$ with values in $\cL({\mathfrak 
   A}, {\mathfrak B})$ is a {\em completely positive kernel} if
   \begin{equation}  \label{cpker}
       \sum_{i,j=1}^{N} b_{i}^{*} \Gamma(z_{i}, z_{j}) [a_{i}^{*} 
       a_{j}] b_{j} \ge 0 \text{ (in ${\mathfrak B}$)}
   \end{equation}
   for all choices of $z_{1}, \dots, z_{N} \in \Omega$, $a_{1}, 
   \dots, a_{N} \in {\mathfrak A}$, $b_{1}, \dots, b_{N} \in 
   {\mathfrak B}$ for all $N = 1,2, \dots$.  The following 
   characterization of completely positive kernels is the completely 
   positive parallel to Theorems \ref{T:posker} and \ref{T:RKC*mod}.
   
   \begin{theorem}  \label{T:cpker} (See \cite{BBLS, BBFtH}.)
       Given a function $\Gamma$ on $\Omega \times \Omega$ with 
       values in $\cL({\mathfrak A}, {\mathfrak B})$, the following 
       are equivalent:
       \begin{enumerate}
       \item $\Gamma$ is a completely positive kernel, i.e., 
       condition \eqref{cpker} holds.
       
       \item There is an $({\mathfrak A}, {\mathfrak 
       B})$-correspondence $\bcH(\Gamma)$ whose elements consist of 
       $\cB$-valued functions $f$ on $\Omega$ such that $K(\cdot, 
       w)[a] \in \bcH(\Gamma)$ for each $w \in \Omega$ and $a \in 
       {\mathfrak A}$ and such that
       $$
       \langle f, K(\cdot, w)[a] \rangle_{\bcH(\Gamma)} = \left( 
       a^{*} \cdot f\right)(w)
       $$
       for all $f \in \bcH(\Gamma)$, $a \in {\mathfrak A}$, and $w \in 
       \Omega$.
       
       \item $K$ has a Kolmogorov decomposition of the following 
       form:  there is an $({\mathfrak A}, {\mathfrak 
       B})$-correspondence $\bcH$ and a function $H$ on $\Omega$ with 
       values in the space $\cL(\bcH, {\mathfrak B})$ of adjointable 
       operators from $\bcH$ to ${\mathfrak B}$ so that
       $$
         K(z,w)[a] = H(z) \pi(a) H(w)^{*}.
    $$
    Here $a \mapsto \pi(a)$ represents the left ${\mathfrak 
    A}$-action on $\bcH$:  $\pi(a) f = a \cdot f$ for $f \in \bcH$.
       \end{enumerate}
    In case ${\mathfrak B} = \cL(\cE)$ for a Hilbert space $\cE$, 
    then we also have Hilbert space versions of conditions (2) and 
    (3):
    \begin{enumerate}
	\item[(2$^{\prime}$)]  There is an 
	$({\mathfrak A}, {\mathbb C})$-correspondence  $\cH(\Gamma)$ 
	(i.e., a Hilbert space	$\cH(\Gamma)$ equipped with a $*$-representation $\pi \colon 
	{\mathfrak A} \to \cL(\cH(\Gamma))$ of 
	${\mathfrak A}$) whose elements are $\cE$-valued functions $f$ 
	on $\Omega$ such that $K(\cdot, w)[a] e \in \cH(\Gamma)$ for 
	each $w \in \Omega$, $a \in {\mathfrak A}$, $e \in \cE$, and 
	such that
	$$
	\langle f, K(\cdot, w)[a]e \rangle_{\cH(\Gamma)} = \langle 
	\left( a^{*} \cdot f \right)(w), \, e \rangle_{\cE}
	$$
	for all $f \in \cH(\Gamma)$, $a \in {\mathfrak A}$, $w \in 
	\Omega$.
	
	\item[(3$^{\prime}$)]  There exists a Hilbert space $\cH$ 
	carrying a $*$-representation $\pi$ of ${\mathfrak A}$ and 
	there exists a function $H \colon \Omega \to \cL(\cH, \cE)$ 
	so that
	$$
	K(z,w)[a] = H(z) \pi(a) H(w)^{*}.
	$$
	\end{enumerate}
       \end{theorem}
       
    \begin{remark}  \label{R:RSbounded}  
      {\em The positivity condition in 
      Theorem \ref{T:RSbounded} can be equivalently formulated as the 
      condition that the kernel
      $$
      k_{\Gamma, S, K}(z,w) = \left[ \Gamma(z,w)[I - S(w)^{*} S(z) 
      ],\, K(z,w) \right]_{\cC_{1}(\cE) \times \cL(\cE)}
      $$
 be a positive ${\mathbb C}$-valued kernel on $\Omega$ for every 
 choice of completely positive kernel 
 $$
 \Gamma \colon \Omega \times \Omega \to\cL(\cL(\cU), \cC_{1}(\cE)),
 $$
 where the outside bracket 
 $$
 [ \cdot, \cdot ]_{\cC_{1}(\cE) \times \cL(\cE)}
 $$ 
 is the duality pairing between the trace-class operators $\cC_{1}(\cE)$ and 
 the bounded linear operators $\cL(\cE)$.
 }\end{remark} 
   
   \subsection{$\Psi$-unitary colligations}  
   \label{S:colligations}
   
   For the transfer-function realization
   $$
   S(z) = D + z C (I - zA)^{-1}B
   $$
   in the operator-valued test-function setting to be developed in the 
   sequel, we shall need a more elaborate version of the unitary 
   colligation matrix $\bU = \left[ \begin{smallmatrix} A & B \\ C & 
   D \end{smallmatrix} \right]$ which we now describe.  Given a 
   collection of test functions $\Psi$ as in Section \ref{S:test}, as 
   described there we view $\Psi$ as a completely regular topological 
   space.  Then the space $C_{b}(\Psi, \cL(\cY_{T}))$ of bounded 
   $\cL(\cY_{T})$-valued functions on $\Psi$ is a $C^{*}$-algebra 
   while the space $C_{b}(\Psi, \cL(\cY_{T}, \cU_{T}))$ of continuous 
   $\cL(\cY_{T}, \cU_{T})$-valued functions is not (unless $\cU_{T} = 
   \cY_{T}$).  However we may view $C_{b}(\Psi, \cL(\cY_{T}, 
   \cU_{T}))$ as a $(C_{b}(\Psi,, \cU_{T}), C_{b}(\Psi, 
   \cY_{T}))$-corres\-pondence, with $C_{b}(\Psi, \cL(\cY_{T}))$-valued 
   inner product given by
   $$
   \left( \langle F, \, F' \rangle_{C_{b}(\Psi, \cL(\cY_{T}, 
   \cU_{T}))} \right) (\psi) : = F'(\psi)^{*} F(\psi).
   $$
   If $\cX$ is a Hilbert space carrying a $*$-representation $\rho$ of 
   $C_{b}(\Psi, \cL(\cY_{T}))$, then we may view $\cX$ as a 
   $(C_{b}(\Psi, \cL(\cY_{T})), {\mathbb C})$ correspondence (with 
   the representation $\rho$ providing the left $C_{b}(\Psi, 
   \cL(\cY_{T}))$-action on $\cX$) and form 
   the internal tensor product $C_{b}(\Psi, \cL(\cY_{T}, \cU_{T})) 
   \otimes_{\rho} \cX$.  We shall say that a  $2 \times 2$-block 
   unitary matrix $\bU = \left[ \begin{smallmatrix} A & B \\ C & D 
   \end{smallmatrix} \right]$ is a {\em $\Psi$-unitary colligation} 
   if $\bU$ has the form
   $$
   \bU = \begin{bmatrix} A & B  \\ C & D \end{bmatrix} \colon 
   \begin{bmatrix} \cX \\ \cU \end{bmatrix} \to \begin{bmatrix} 
       C_{b}(\Psi, \cL(\cY_{T}, \cU_{T})) \otimes_{\rho} \cX \\ \cY 
       \end{bmatrix}
    $$
    for $X$ equal to a Hilbert space carrying a $*$-representation 
    $\rho$ of $C_{b}(\Psi, \cL(\cY_{T}))$.
    
    A particular element of $C_{b}(\Psi, \cL(\cY_{T}, \cU_{T}))$ is the 
    function ${\mathbb E}(z)^{*}$, where 
    ${\mathbb E}(z)$ is as in \eqref{boldE} (for a given $z \in 
    \Omega$).  Hence the tensor 
    multiplication operator
    \begin{equation}   \label{LboldE*}
	L_{{\mathbb E}(z)^{*}} \colon x \mapsto {\mathbb E}(z)^{*} 
	\otimes x
   \end{equation}
   defines an operator from $\cX$ to $C_{b}(\Psi, \cL(\cY_{T}, 
   \cU_{T})) \otimes_{\rho} \cX$; one can verify that its adjoint 
   acting on pure tensors is given by
   $$
    L^{*}_{{\mathbb E}(z)^{*}} \colon g \otimes x \mapsto 
    \rho({\mathbb E}(z) g) x.
   $$
   As a consequence we get the identity
   \begin{equation}   \label{tensorid}
   L^{*}_{{\mathbf E}(z)^{*}} L_{{\mathbb E}(w)^{*}} x =
   L^{*}_{{\mathbb E}(z)^{*}} ({\mathbb E}(w)^{*} \otimes x) = \rho\left( 
   {\mathbb E}(z) {\mathbb E}(w)^{*} \right) x.
   \end{equation}
   In case $\cY_{T} = \cU_{T}$ (the square case),  then $C_{b}(\Psi, 
   \cL(\cY_{T}, \cU_{T}) \otimes_{\rho} \cX$ collapses down to $\cX$ 
   (a consequence of the balancing property \eqref{balance}), and 
   then $L^{*}_{{\mathbb E}(z)^{*}}$ can be identified with 
   $L^{*}_{{\mathbb E}(z)^{*}} = \rho({\mathbb E}(z))$.  We conclude 
   that the tensor-product construction is exactly the technical tool 
   needed to push the square case to the non-square case.
   This type of colligation matrix appears in \cite{Ambrozie04, 
   DMM07, DM07} for the square case and in \cite{MS08} for the 
   nonsquare case.

   \section{The Schur-Agler class associated with a collection of 
   test functions}  \label{S:main}
   
   Suppose that we are given a collection $\Psi$ of test functions 
   $\psi \colon \Omega \to \cL(\cU_{T}, \cY_{T})$ satisfying the 
   admissibility condition \eqref{test-axiom}.  For $\cE$ any 
   auxiliary HIlbert space and $K$ an $\cL(\cE)$-valued positive 
   kernel on $\Omega$, we say that $K$ is {\em $\Psi$-admissible}, 
   written as $K \in \cK_{\Psi}(\cE)$, if 
   the operator $R_{\psi} \colon f(z) \mapsto f(z) \psi(z)$ is 
   contractive from $\cH(K)_{\cY_{T}}$ to $\cH(K)_{\cU_{T}}$ for each 
   $\psi \in \Psi$, or equivalently (by Theorem \ref{T:RSbounded}), if
   the ${\mathbb C}$-valued kernel
   \begin{equation}   \label{kXpsiK}
   k_{X, \psi, K}(z,w) = \operatorname{tr}\left(X(w)^{*} (I - \psi(w)^{*} 
   \psi(z)) X(z) K(z,w)\right) 
   \end{equation}
   is a positive kernel for each choice of $X \colon \Omega \to 
   \cC_{2}(\cE, \cU_{T})$ and $\psi \in \Psi$.
   We then say that the function $S \colon \Omega \to \cL(\cU, \cY)$ 
   is in the $\Psi$-Schur-Agler class $\mathcal{SA}_{\Psi}(\cU, \cY)$
   if the operator $R_{S}$ of right multiplication by $S$ is 
   contractive from $\cH(\cY)_{\cY}$ to $\cH(\cY)_{\cU}$ for each 
   $\Psi$-admissible $\cL(\cY)$-valued positive kernel $K$, or 
   equivalently, if the kernel
   \begin{equation}   \label{kYSK}
   k_{Y,S,K}(z,w) = \operatorname{tr}(Y(w)^{*} (I - S(w)^{*} S(z)) 
   Y(z) K(z,w))
   \end{equation}
   is a positive ${\mathbb C}$-valued kernel for each choice of $Y 
   \colon \Omega \to \cC_{2}(\cY, \cU)$ and  $K  \in \cK_{\Psi}(\cY)$.
   
   Our main result on the Schur-Agler class $\mathcal{SA}_{\Psi}(\cU, 
   \cY)$ is the following.
   
   \begin{theorem}  \label{T:SchurAgler}
       Suppose that we are given a collection of test functions 
       $\Psi$ satisfying condition \eqref{test-axiom} and $S_{0}$ is a 
       function on some subset  $\Omega_{0}$  of $\Omega$ with values in $\cL(\cU, \cY)$.  
       Consider the following conditions:
       
       \begin{enumerate}
	   \item $S_{0}$ can be extended to a function $S$ 
	   defined on all of $\Omega$ such that $S \in \mathcal{SA}_{\Psi}(\cU, \cY)$, i.e., the 
	   kernel \eqref{kYSK} is a positive kernel for all choices 
	   of  $\cL(\cY, \cU)$-valued functions $Y$ on $\Omega_{0}$ and all choices of 
	   kernels $K \in \cK_{\Psi}(\cY)$.
	   
	   \item $S_{0}$ has an {\em Agler decomposition} on 
	   $\Omega_{0}$, i.e., there is a completely positive kernel 
	   $\Gamma \colon \Omega_{0} \times \Omega_{0} \to 
	   \cL(C_{b}(\Psi, \cL(\cY_{T})), \cL(\cY))$ so that 
	   \begin{equation}   \label{Aglerdecom}
	   I - S_{0}(z) S_{0}(w)^{*} = 
	   \Gamma(z,w)[ I - {\mathbb E}(z) {\mathbb E}(w)^{*} ]
	   \end{equation}
	   for all $z,w \in \Omega_{0}$ (where ${\mathbb E}(z) \in 
	   C_{b}(\Psi, \cL(\cU_{T}, \cY_{T}))$ is as in \eqref{boldE}).
	   
	   \item There is a Hilbert state space $\cX$ which carries a 
	   $*$-representation of the $C^{*}$-algebra $C_{b}(\Psi, \cL(\cY_{T}))$ and a  
	   $\Psi$-unitary colligation $\bU$  (see 
	   Section \ref{S:colligations}) 
	   \begin{equation}   \label{bUcol}
	   \bU = \begin{bmatrix} A & B \\ C & D \end{bmatrix} \colon 
	   \begin{bmatrix} \cX \\ \cU \end{bmatrix} \to 
	       \begin{bmatrix} C_{b}(\Psi, \cL(\cY_{T}, \cU_{T})) 
		   \otimes_{\rho} \cX \\ \cY \end{bmatrix} 
	 \end{equation}
	 so that $S_{0}$ has the transfer-function realization
	 \begin{equation}  \label{S0realization}
	 S_{0}(z) = D + C (I - L^{*}_{{\mathbb E}(z)^{*}} A)^{-1}  
	 L^{*}_{{\mathbb E}(z)^{*}} B
	 \end{equation}
	 for $z \in \Omega_{0}$.
	 \end{enumerate}
	 Then (1) $\Rightarrow$  (2)   $\Leftrightarrow$ (3); if 
	 $\operatorname{dim} \cY_{T} < \infty$, then also (2) 
	 $\Rightarrow$ (1) and hence (1), (2), (3) are all equivalent 
	 to each other.
	 \end{theorem}
	 
	We shall prove (1) $\Rightarrow$ (2) $\Rightarrow$ (3) 
	$\Rightarrow$ (2) and, if $\operatorname{dim} \cY_{T} < 
	\infty$, then also (2) $\Rightarrow$ (1).
	 
	 \begin{proof}[Proof of (1) $\Rightarrow$ (2):] 
	     
	  \textbf{Step 1:   $\Omega_{0}$ is a finite 
	    subset of $\Omega$.} 
	    
	    We define a cone $\cC$ by
	    \begin{align*} 
	    \cC =  & \{ \Xi \colon \Omega_{0} \times \Omega_{0} \to 
	    \cL(\cY) \colon 
	     \Xi(z,w)  = \Gamma(z,w)[ I - {\mathbb E}(z) {\mathbb 
	    E}(w)^{*}] \text{ for some } \\ & \text{completely positive 
	    kernel } 
	     \Gamma \colon \Omega_{0} \times \Omega_{0} \to
	    \cL(C_{b}(\Psi, \cL(\cY_{T})), \cL(\cY))\}.  
	    \end{align*}
	    Note that the elements of $\cC$ can be viewed as matrices 
	    with rows and columns indexed by the finite set 
	    $\Omega_{0}$ and matrix entries in $\cL(\cY)$.  Thus we 
	    may view $\cC$ as a subset of the linear space $\cV$ of 
	    all such matrices with topology of pointwise weak-$*$ 
	    convergence.  We shall need a few preliminary 
	    lemmas.  It is easy to verify that $\cC$ is a cone in 
	    $\cV$. 
	    
	    \begin{lemma} \label{L:Cclosed} The cone $\cC$ is closed 
		in $\cV$.
	  \end{lemma}
	  
	  \begin{proof}[Proof of Lemma]  Suppose that $\{ 
	      \Xi_{\alpha}\}$ is a net of elements of $\cC$ such that 
	      $\{\Xi_{\alpha}(z,w)\}$ converges weak-$*$ to 
	      $\Xi(z,w)$ for each $z,w \in \Omega_{0}$.    Thus, for 
	      each index $\alpha$ there is a choice of completely 
	      positive kernel $\Gamma_{\alpha}$ so that
	      \begin{equation}   \label{alpha-rep}
		  \Xi_{\alpha}(z,w) = \Gamma_{\alpha}(z,w)[I - 
		  {\mathbb E}(z) {\mathbb E}(w)^{*} ].
	  \end{equation}
	  The computation
	  \begin{align*}
	      \Gamma_{\alpha}(z,z)[I] & = 
	      \Gamma_{\alpha}(z,z)[ (I - {\mathbb E}(z) {\mathbb E}(z)^{*})^{1/2} 
	      (I - {\mathbb E}(z) {\mathbb E}(w)^{*})^{-1}
	      (I - {\mathbb E}(z) {\mathbb E}(z)^{*})^{1/2} ] \\
	      & \le \Gamma_{\alpha}(z,z)\left[(I - {\mathbb E}(z) {\mathbb 
	      E}(z)^{*})^{1/2}\left( \frac{1}{1 - \|{\mathbb E}(z) 
	      \|^{2}}\right) (I - {\mathbb E}(z) {\mathbb 
	      E}(z)^{*})^{1/2} \right] \\\
	      & = \left( \frac{1}{1 - \|{\mathbb E}(z) \|^{2}} 
	      \right) \Gamma_{\alpha}(z,z) [ I - {\mathbb E}(z) 
	      {\mathbb E}(z)^{*} ]  \\
	      & = \left( \frac{1}{1 - \|{\mathbb E}(z) \|^{2}} 
	      \right)  \Xi_{\alpha}(z,z)
	   \end{align*}
	   shows that
	   \begin{equation}  \label{est1}
	       \| \Gamma_{\alpha}(z,z) \| \le M_{z} \| 
	       \Xi_{\alpha}(z,z) \| \text{ where } M_{z} = \frac{1}{1 
	       - \| {\mathbb E}(z) \|^{2}},
	     \end{equation}
	     where we used here the underlying assumption 
	     \eqref{test-axiom} for our set of test functions 
	     $\Psi$.  Since the block $2 \times 2$ matrix
	     $$
	     \begin{bmatrix} \Gamma_{\alpha}(z,z)[I] & 
		 \Gamma_{\alpha}(z,w)[I] \\
	 \Gamma_{\alpha}(w,z)[I]   & \Gamma_{\alpha}(w,w)[I]  
     \end{bmatrix}
     $$
     is positive semidefinite for each index $\alpha$ and each pair 
     of points $z,w \in \Omega_{0}$,  it follows that
     \begin{equation}   \label{est2}
	 \| \Gamma_{\alpha}(z,w) \| \le M_{z} M_{w} \| 
	 \Xi_{\alpha}(z,w) \|^{1/2} \| \Xi_{\alpha}(w,w) \|^{1/2}.
  \end{equation}
  Since $\Omega_{0}$ is finite, we see that $\| 
  \Gamma_{\alpha}(z,w)\|$ is in fact bounded uniformly with respect 
  to the indices $\alpha$ and the points $z,w$ in $\Omega_{0}$.  
  Since $\cL(C_{b}(\Psi, \cL(\cY_{T})), \cL(\cY))$ is the 
  Banach-space dual of the projective tensor-product  Banach  space 
  $\cC_{1}(\cY) \otimes C_{b}(\Psi, \cL(\cY_{T}))$ (see e.g.~ 
  \cite[Theorem IV.2.3]{Takesaki}),  it follows from the 
  Banach-Alaoglu theorem that there is a subnet $\{\Gamma_{\beta}\}$ 
  of $\{ \Gamma_{\alpha}\}$ such that $\{\Gamma_{\beta}(z,w)\}$ converges 
  weak-$*$ to some $\Gamma_{\infty}(z,w) \in \cL(C_{b}(\Psi, 
  \cL(\cY_{T})), \cL(\cY))$.  It is straightforward to verify that 
  the defining property \eqref{cpker} for a completely positive 
  kernel is preserved under such weak-$*$ limits; hence 
  $\Gamma_{\infty}$ is again a completely positive kernel.  Moreover, 
  from the fact that $\{\Xi_{\alpha}(z,w)\}$ converges weak-$*$ to 
  $\Xi(z,w)$, we get that the subnet  $\{\Xi_{\beta}(z,w)\}$ also 
  converges weak-$*$ to $\Xi(z,w)$.  Taking limits in the formula 
  \eqref{alpha-rep} leads us to the representation 
  $$
    \Xi(z,w) = \Gamma_{\infty}(z,w)[ I - {\mathbb E}(z) {\mathbb 
    E}(w)^{*} ]
  $$
  for the limit kernel $\Xi(z,w)$.  We conclude that the limit kernel 
  $\Xi$ is again in $\cC$ as wanted.
 \end{proof}

	  \begin{lemma}  \label{L:posdefin}
	      Suppose that $\Xi(z,w) = H(z) H(w)^{*}$ is a positive 
	      $\cL(\cY)$-valued
	      kernel on $\Omega_{0}$.  Then $\Xi$ is in $\cC$.
	      \end{lemma}
	      
	      \begin{proof}[Proof of Lemma]
		  Let us say that $\Xi(z,w) = H(z) H(w)^{*}$ where $H 
		  \colon \Omega \to \cL(\cX, \cY)$ for some 
		  coefficient Hilbert space $\cX$.
		  Let $\psi_{0}$ be any fixed test function in 
		  $\Psi$.  It suffices to find another coefficient 
		  Hilbert space $\widetilde \cX$ and a function $G \colon \Omega_{0} 
		  \to \cL(\widetilde \cX \otimes \cY_{T}, \cY)$ so that
	$$
	\Xi(z,w) = G(z) \left( I_{\widetilde \cX} \otimes (I - \psi_{0}(z) \psi_{0}(w)^{*}) 
	 \right) G(w)^{*},
	$$
	for then we have the needed representation $\Xi(z,w) = 
	\Gamma_{0}(z,w)[ I - {\mathbb E}(z) {\mathbb E}(w)^{*}]$ with 
	$\Gamma_{0}$ given by 
	$$
	\Gamma_{0}(z,w)[g] = G(z) ( I_{\widetilde \cX} \otimes g(\psi_{0})) G(w)^{*}.
	$$
	Toward this end, choose a unit vector $y_{0}$ in $\cY_{T}$ 
	and note that
	$$
	y_{0}^{*} (I - \psi_{0}(z) \psi_{0}(w)^{*}) y_{0} = 1 - 
	y_{0}^{*} \psi_{0}(z) \psi_{0}(w)^{*} y_{0}
	$$
	is invertible (as an element of ${\mathbb C}$) by our 
	underlying assumption \eqref{test-axiom}.  Moreover we have 
	the geometric series representation for the inverse:
	\begin{equation}   \label{frac-rep1}
	    \frac{1}{1 - y_{0}^{*} \psi_{0}(z) \psi_{0}(w)^{*} y_{0}} 
	    =
	    \sum_{n=0}^{\infty} \left( y_{0}^{*}\psi_{0}(z) 
	    \psi_{0}(w)^{*} y_{0} \right)^{n}
	 \end{equation}
	 where each term $\left( y_{0}^{*} \psi_{0}(z) 
	 \psi_{0}(w)^{*} y_{0} \right)^{n}$ is a positive kernel due 
	 to the Schur multiplier theorem (see e.g.~\cite[Theorem 
	 3.7]{Paulsen}).  Thus there exist functions  $g_{n} \colon 
	 \Omega_{0} \to \cL(\widetilde \cG_{n}, {\mathbb C})$ so that
	 $$
	 \left( y_{0}^{*} \psi_{0}(z) \psi_{0}(w)^{*} y_{0} 
	 \right)^{n} = g_{n}(z) g_{n}(w)^{*}.
	 $$
	 Then we may rewrite \eqref{frac-rep1} as
	 \begin{equation}   \label{frac-rep2}
	     \frac{1}{1 - y_{0}^{*} \psi_{0}(z) \psi_{0}(w)^{*} 
	     y_{0}} = \sum_{n=0}^{\infty} g_{n}(z) g_{n}(w)^{*}.
	 \end{equation}
	 We conclude that
	 \begin{align*}
	     \Xi(z,w) & = H(z) H(w)^{*} \\
	     & = H(z) \left( \frac{1}{1 - y_{0}^{*} \psi_{0}(z) 
	     \psi_{0}(w)^{*} y_{0}} \cdot (1 - y_{0}^{*} \psi_{0}(z) 
	     \psi_{0}(w)^{*} y_{0}) I_{\cX} \right) H(w)^{*}  \\
 & = \sum_{n=0}^{\infty} H(z) \left( g_{n}(z) g_{n}(w)^{*} 
 (1 - y_{0}^{*} \psi_{0}(z)  \psi_{0}(w)^{*} y_{0}) I_{\cX} \right) 
 H(w)^{*}  \\  
 & = \sum_{n=0}^{\infty} H(z) g_{n}(z) \left( (1 - y_{0}^{*} 
 \psi_{0}(z) \psi_{0}(w)^{*} y_{0}) I_{\widetilde \cG_{n}} \right) 
 g_{n}(w)^{*} H(w)^{*}  \\
 & = \sum_{n=0}^{\infty} H(z) (g_{n}(z) \otimes y_{0}^{*}) \left( 
 I_{\widetilde \cG_{n}} \otimes (I - \psi_{0}(z) \psi_{0}(w)^{*}) 
 \right)  (g_{n}(w)^{*} \otimes y_{0}) H(w)^{*}  \\
 & = G(z) \left( I_{\widetilde \cX} \otimes (I - \psi_{0}(z) 
 \psi_{0}(w)^{*} ) \right) G(w)^{*}
  \end{align*}
  where we set 
  $$
  G(z) = \begin{bmatrix} H(z) (g_{1}(z) \otimes y_{0}^{*}) &
   H(z) (g_{2}(z) \otimes y_{0}^{*})  & \cdots \end{bmatrix}, \quad
   \widetilde \cX = \bigoplus_{n=1}^{\infty} \widetilde \cG_{n}.
 $$

	\end{proof}
	      
	      Let us now note that the assertion of the condition (2) 
	      in the statement of the Theorem is that the kernel 
	      $\Xi_{S_{0}}(z,w) : = I - S_{0}(z) S_{0}(w)^{*} $ is in 
	      $\cC$.  As $\cV$ is a locally convex  linear topological vector 
	      space and $\cC$ is closed in $\cV$, by a standard 
	      Hahn-Banach separation principle (see \cite[Theorem 
	      3.49b)]{Rudin}), to show that $\Xi_{S} \in \cC$ it 
	      suffices to show:  {\em $\R {\mathbb L}(\Xi_{S}) \ge 0$ 
	      whenever ${\mathbb L}$ is a continuous linear 
	      functional on $\cV$ such that $\R {\mathbb L}(\Xi) \ge 
	      0$ for each $\Xi \in \cC$}.  
	      
	      With this strategy in mind let us suppose that 
	      ${\mathbb L}$ is a continuous linear functional on 
	      $\cV$ such that $\R {\mathbb L}(\Xi) \ge 0$ for each 
	      $\Xi \in \cC$.  We then define ${\mathbb L}_{1}$ on 
	      $\cV$ by
	      $$
	       {\mathbb L}_{1}(\Xi) = \frac{1}{2} \left( {\mathbb 
	       L}(\Xi) + \overline{ {\mathbb L}(\Xi^{\vee})} \right)
	      $$
	      where we set
	      $$
	      \Xi^{\vee}(z,w) = \Xi(w,z)^{*}.
	      $$
	      Easy properties are that
	      \begin{equation}  \label{easy1}
	  {\mathbb L}_{1}(\Xi) = \R {\mathbb L}(\Xi) \text{ if } 
	  \Xi^{\vee} = \Xi.
	  \end{equation}
	 
 For $\epsilon > 0$ be an arbitrarily small but positive 
	  number, we use the functional ${\mathbb L}_{1}$ to define 
	  an inner product on the space $\cH_{{\mathbb L}_{1}, 
	  \epsilon}$ of functions $f \colon \Omega_{0} \to \cY$ by
	  $$
	  \langle f, g \rangle_{\cH_{{\mathbb L}_{1, \epsilon}}}
	  = {\mathbb L}_{1}(\Delta_{f,g}) + \epsilon^{2} 
	  \sum_{w \in \Omega_{0}} 
	  \operatorname{tr}\left(\Delta_{f,g}(w,w) \right)
	  $$
	  where we have set
	  \begin{equation}   \label{Deltafg}
	  \Delta_{f,g}(z,w) = f(z) g(w)^{*}.
	  \end{equation}
	  By Lemma \ref{L:posdefin} we know that $\Delta_{f,f} \in 
	  \cC$ and hence $\R {\mathbb L}(\Delta_{f,f}) \ge 0$.  
	  Since $\Delta_{f,f} = \Delta^{\vee}_{f,f}$, as a 
	  consequence of \eqref{easy1} we know that $\R {\mathbb 
	  L}(\Delta_{f,f}) = {\mathbb L}_{1}(\Delta_{f,f})$.  From 
	  these observations it follows that $\langle \cdot, \cdot 
	  \rangle_{\cH_{{\mathbb L}_{1, \epsilon}}}$ is a positive 
	  semidefinite inner product.  Hence we can take the 
	  Hausdorff completion of $\cH_{{\mathbb L}_{1}, \epsilon}$ 
	  to arrive at a Hilbert space, still denoted as 
	  $\cH_{{\mathbb L}_{1}, \epsilon}$. 
	  
	  For $\cX$ a coefficient Hilbert space, we shall be 
	  interested in the space $\cH_{{\mathbb L}_{1}, \epsilon} 
	  \otimes \cC_{2}(\cX, {\mathbb C})$.  The following lemma is 
	  crucial.
	  
	  \begin{lemma}  \label{L:HLtensor}
	      The space $\cH_{{\mathbb L}_{1}, \epsilon} \otimes 
	      \cC_{2}(\cX, {\mathbb C})$ can be identified with the 
	      space $(\cH_{{\mathbb L}_{1}, \epsilon})_{\cX}$ 
	      consisting of $\cC_{2}(\cX, \cY)$-valued functions $f$ on 
	      $\Omega$ with inner product given by
	      \begin{equation}   \label{Hepsilon}
	      \langle f, g \rangle_{\cH_{{\mathbb L}_{1}, 
	      \epsilon})_{\cX}} = {\mathbb L}_{1}(\Delta_{f,g}) +
	      \epsilon^{2} \sum_{w \in \Omega_{0}} 
	      \operatorname{tr}\left( \Delta_{f,g}(w,w) \right)
	      \end{equation}
	      where $\Delta_{f,g}$ has the same form as in 
	      \eqref{Deltafg} (but where now the middle space is 
	      $\cX$ rather than ${\mathbb C}$):
	      $$
	      \Delta_{f,g}(z,w) = f(z) g(w)^{*}.
	      $$
	      \end{lemma}
	      
	      \begin{proof}[Proof of lemma]
		  For convenience of notation we drop the 
		  $\epsilon$-term in the inner product as the 
		  $\epsilon > 0$ case proceeds in the same way but 
		  with more cumbersome notation.  For $f \otimes 
		  x^{*}$ a pure tensor in ${\mathcal H}_{{\mathbb 
		  L}_{1}} \otimes \cC_{2}(\cX. {\mathbb C})$ (so $f 
		  \in \cH_{{\mathbb L}_{1}}$ and $x \in \cX \cong 
		  \cL({\mathbb C}, \cX)$) and similarly for $f' 
		  \otimes x^{\prime *}$, we have
		  \begin{align*}
 \langle f \otimes x^{*},\, f' \otimes x^{\prime *} 
 \rangle_{\cH_{{\mathbb L}_{1}} \otimes C_{2}(\cX, {\mathbb C})} & =
 \left\langle \langle f, f' \rangle_{\cH_{{\mathbb L}_{1}}} x^{*},\, 
 x^{\prime *} \right\rangle_{\cC_{2}(\cX, {\mathbb C})}   \\
 & = {\mathbb L}_{1}(\Delta_{f,f'}) x^{*} x^{\prime *} =
 {\mathbb L}_{1}(\Delta_{f,f'} x^{*} x') 
 \end{align*}
 where the last step follows since 
 $x^{*} x'$ is just a complex number. Next observe that
 \begin{align*}
     \Delta_{f,f'}(z,w) x^{*} x & = f(z) f'(w)^{*} (x^{*} x') = f(z) 
     (x^{*} x') f'(w)^{*} \\
     & = \left(f(z) x^{*}\right) \left(f'(w) x^{\prime *} \right)^{*}
     = \Delta_{f \cdot x^{*}, f' \cdot x^{\prime *}}(z,w).
 \end{align*}
 By extending this calculation to linear combinations of pure 
 tensors, the result follows.
 \end{proof}
 
 With the formulation of the space $\left(\cH_{{\mathbb L}_{1}, 
 \epsilon}\right)_{\cX}$ in hand, it makes sense to ask whether the 
 right multiplication operator $R_{\psi} \colon f(z) \mapsto f(z) 
 \psi(z)$ defines a contraction operator from $\left(\cH_{{\mathbb 
 L}_{1},\epsilon}\right)_{\cY_{T}}$ to $\left( \cH_{{\mathbb L}_{1}, 
 \epsilon}\right)_{\cU_{T}}$.  The answer is given by the next lemma.
 
 \begin{lemma}   \label{L:Rpsi}
     For each test function $\psi \in \Psi$, the right multiplication 
     operator $R_{\psi}$ defines a contraction operator form  $\left(\cH_{{\mathbb 
 L}_{1},\epsilon}\right)_{\cY_{T}}$ to $\left( \cH_{{\mathbb L}_{1}, 
 \epsilon}\right)_{\cU_{T}}$.
 \end{lemma}
 
 \begin{proof}[Proof of Lemma]  $R_{\psi}$ is contractive if and only 
     if
     $$
     \|f\|^{2}_{(\cH_{{\mathbb L}_{1}, \epsilon})_{\cY_{T}}} -
      \|R_{\psi} f\|^{2}_{(\cH_{{\mathbb L}_{1}, 
      \epsilon})_{\cU_{T}}} \ge 0
     $$
 for all $f \in \left( \cH_{{\mathbb L}_{1}, \epsilon}\right)_{\cY_{T}}$.
 This translates to the condition that
 $$
   {\mathbb L}_{1}(\Delta_{f,f} - \Delta_{f \psi, f \psi}) + 
   \epsilon^{2}  \sum_{w  \in \Omega_{0}}\left[ \Delta_{f,f}(w,w) -
   \Delta_{f \psi, f \psi}(w,w) \right] \ge 0
 $$
 for all such $f$.  Observe that
 $$
 \Delta_{f,f}(z,w) - \Delta_{f \psi, f \psi}(z,w) = 
 f(z) (I - \psi(z) \psi(w)^{*}) f(w)^{*}
 $$
 from which we see that the kernel $\Xi: = \Delta_{f,f} - \Delta_{f \psi, f \psi}$ is 
 in the cone $\cC$:  note that the kernel $\Gamma(z,w)[g] = f(z) g(\psi) f(w)^{*}$
 is completely positive since its Kolmogorov decomposition (condition 
 (3$^{\prime}$) in Theorem \ref{T:cpker}) is exhibited. Thus 
 $\R {\mathbb L}(\Xi) \ge 0$, and hence, since $\Xi = \Xi^{\vee}$, 
 also ${\mathbb L}_{1}(\Xi) \ge 0$.  The $\epsilon$-term is also 
 nonnegative since $\| \psi(w)\| < 1$ for each $w \in \Omega_{0}$.  
 It now follows that $\| R_{\psi} \| \le 1$ as asserted.  
     \end{proof}
     
 To make use of the hypothesis that $S \in \mathcal{SA}_{\Psi}(\cU, 
 \cY)$, we need to convert the space $\cH_{{\mathbb L}_{1}, 
 \epsilon}$ to a reproducing kernel space.  This is done as follows; 
 it is at this point that we make use of the 
 $\epsilon$-regularization of the ${\mathcal H}_{{\mathbb 
 L}_{1}}$-inner product.  
 
 \begin{lemma}  \label{L:HLtoHK}
     The space $(\cH_{{\mathbb L}_{1}, \epsilon})_{\cY}$ is 
     isometrically equal to a reproducing kernel Hilbert spaces 
     $\cH(K)_{\cY}$ for a positive kernel $K \in \cK_{\Psi}(\cY)$.
     \end{lemma}
     
     \begin{proof}[Proof of lemma]  We wish to apply Theorem 
	 \ref{T:HKXconverse} with $\cE$ and $\cX$ equal to $\cY$ and 
	 with $\Omega_{0}$ equal to $\Omega$.
	 To this end, we note that elements of $\left( \cH_{{\mathbb 
	 L}_{1}, \epsilon} \right)_{\cY}$ are $\cC_{2}(\cY)$-valued 
	 functions, at least on the dense set before the 
	 Hausdorff-completion step is carried out in the construction 
	 of the space.  However, the presence of the term with the 
	 $\epsilon^{2}$ factor in the definition of the $\left( 
	 \cH_{{\mathbb L}_{1}, \epsilon}\right)_{\cY}$-inner product 
	 guarantees that the point-evaluation map $\ev_{w} \colon 
	 \left( \cH_{{\mathbb L}_{1}, \epsilon}\right)_{\cY} \to 
	 \cC_{2}(\cY)$ is bounded with norm at most $1/\epsilon$.
	 Hence condition (1) in Theorem \ref{T:HKXconverse} is 
	 verified.  Condition (2) is straightforward since 
	 $\left(\cH_{{\mathbb L}_{1}, \epsilon}\right)_{\cY}$ is 
	 itself a tensor-product space $\cH_{{\mathbb L}_{1}, 
	 \epsilon} \otimes \cC_{2}(\cY, {\mathbb C})$.
	 We conclude that $\left(\cH_{{\mathbb L}_{1}, 
	 \epsilon}\right)_{\cY}$ is isometrically equal to a 
	 reproducing kernel Hilbert space $\cH(K)_{\cY}$ for a uniquely 
	 determined $\cL(\cY)$-valued positive kernel $K$.
	 
	 Finally we must verify that $K$ is $\Psi$-admissible.  But 
	 this is an immediate consequence of Lemma \ref{L:Rpsi}.
	 \end{proof}
	 
	 To conclude the proof of Step 1 (the case where $\Omega_{0}$ 
	 if finite), we proceed as follows.  Let $K$ be the positive 
	 kernel identified in Lemma \ref{L:HLtoHK}.  Since $K \in 
	 \cK_{\Psi}(\cY)$, we use the assumption that $S$ is in the 
	 Schur-Agler class $\mathcal{SA}_{\Psi}(\cU, \cY)$ to 
	 conclude that the operator $R_{S}$ of right multiplication by 
	 $S$ is contractive from $\cH(K)_{\cY}$ to $\cH(K)_{\cU}$.  As Lemma 
	 \ref{L:HLtoHK} also tells us that $\cH(K)_{\cY}$ is 
	 isometrically equal to $\left(\cH_{{\mathbb L}_{1}, 
	 \epsilon}\right)_{\cY}$, trivially we can also say that 
	 $R_{S}$ is contractive from 
	 $\left(\cH_{{\mathbb L}_{1,\epsilon}}\right)_{\cY}$ to  $\left(\cH_{{\mathbb 
	 L}_{1,\epsilon}}\right)_{\cU}$.  The criterion for this to 
	 be the case is that
	 $$ \| f \|^{2}_{(\cH_{{\mathbb L}_{1}, \epsilon})_{\cY}} -
	 \| R_{S} f \|^{2}_{(\cH_{{\mathbb L}_{1}, \epsilon})_{\cU}}
	 \ge 0 \text{ for all } f \in \left(\cH_{{\mathbb L}_{1}, 
	 \epsilon}\right)_{\cY},
	 $$
	 or, equivalently
	$$
	{\mathbb L}_{1}\left( \Delta_{f,f} - \Delta_{fS_{0}, fS_{0}} \right) 
	 + \epsilon^{2} \sum_{w \in \Omega_{0}} \operatorname{tr}\left(  
	 \Delta_{f,f}(w,w) - \Delta_{fS_{0}, fS_{0}}(w,w) \right) \ge 0 
	 \text{ for all } f,
	 $$
	 where $\Delta_{f,f}(z,w) - \Delta_{fS_{0}, fS_{0}}(z,w) = 
	 f(z) \Xi_{S}(z,w) f(w)^{*}$.  In particular, taking 
	 $f(z) = P_{n}$ for all $z \in \Omega_{0}$ where $\{P_{n}\}$ is 
	 an increasing sequence of finite-rank orthogonal projections 
	 converging strongly to the identity operator $I_{\cY}$  gives us
	 $$
	 {\mathbb L}_{1}(P_{n}\Xi_{S_{0}}P_{n}) + \epsilon^{2} 
	 \sum_{z \in  \Omega_{0}} \operatorname{tr}\left( P_{n} 
	 \Xi_{S_{0}}(z,z) P_{n} \right) \ge 0.
	 $$
	 As this holds for all $\epsilon > 0$, we may take the limit 
	 as $\epsilon \to 0$ (while holding $n$ fixed) to get
	 \begin{equation}   \label{L1Xinpos}
	 {\mathbb L}_{1}(P_{n} \Xi_{S_{0}} P_{n}) \ge 0
	 \end{equation}
	 for all $n$.  By the weak-$*$ continuity of ${\mathbb 
	 L}_{1}$ we have that 
	 $$
	 \lim_{n \to \infty}{\mathbb L}_{1}(P_{n} \Xi_{S_{0}} P_{n}) =
	 {\mathbb L}_{1}(\Xi_{S_{0}}).
	 $$
	 Taking limits in 
	 \eqref{L1Xinpos} then gives us ${\mathbb L}_{1}(\Xi_{S_{0}}) \ge 0$. 
	 As $\Xi_{S_{0}} = \Xi_{S_{0}}^{\vee}$, this gives us finally $\R {\mathbb 
	 L}(\Xi_{S_{0}}) \ge 0$ as required, and we conclude that 
	 $S_{0} \in \cC$ as wanted.  This concludes the proof of Step 
	 1.
	 
	 \smallskip
	 
	 \textbf{Step 2: $\Omega_{0}$ is not necessarily finite.}
	 
	 We now remove that assumption that $\Omega_{0}$ is finite.  
	 It is now understood how this step is efficiently handled as 
	 an application of the Kurosh Theorem (see \cite{DMM07, DM07}).
	 By Step 1, we know that for each finite subset $\Omega_{F}$ 
	 of $\Omega$, there is an associated completely positive 
	 kernel $\Gamma_{F}$ (not necessarily uniquely determined) so that the 
	 Agler decomposition
	 \begin{equation}  \label{AglerdecomF}
	    \Xi_{S_{0}}(z,w) := I - S_{0}(z)S_{0}(w)^{*} = 
	    \Gamma_{\Omega_{F}}(z,w)[ I - {\mathbb E}(z) {\mathbb E}(w)^{*} ]
	 \end{equation}
	 holds for all $z,w \in \Omega_{F}$. 
	 To set up the Kurosh Theorem, for each finite 
	 subset $\Omega_{F} \subset \Omega$, we let $\Phi_{\Omega_{F}}$ denote 
the  collection
$$
\Phi_{\Omega_{F}} = \{ \Xi \colon \Xi \text{ completely 
positive kernel such that \eqref{AglerdecomF} holds for } z,w \in 
\Omega_{F}\}.
$$
By applying the argument used in the proof of Lemma \ref{L:Cclosed}, 
one can see that $\Phi_{\Omega_{F}}$ is compact in the pointwise weak-$*$ 
convergence topology inherited from the space of 
$\cL(C_{b}(\Psi, \cL(\cY_{T})), \cL(\cY))$-valued functions on $\Omega \times \Omega$.  
The Kurosh Theorem (see \cite[page 75]{AP}) tells us that, for each finite 
subset $\Omega_{F}$ of
$\Omega$, there is a choice of completely positive kernel 
$\Gamma_{\Omega_{F}}$ for which \eqref{AglerdecomF} holds on 
$\Omega_{F}$ such that, in addition, whenever $\Omega_{F}, 
\Omega_{F'}$ are two subsets of $\Omega$ with $\Omega_{F} \subset 
\Omega_{F'}$, then $\Gamma_{\Omega_{F'}}|_{\Omega_{F} \times \Omega_{F}} = 
\Gamma_{\Omega_{F}}$.  We may then define a completely positive 
kernel $\Gamma$ on all of $\Omega \times \Omega$ by
$$
  \Gamma(z,w) = \Gamma_{\Omega_{F}}(z,w) \text{ where } \Omega_{F} 
  \text{ finite,} \quad z,w \in \Omega_{F}.
$$
The construction guarantees that $\Gamma$ is well defined and the 
fact that each $\Gamma_{\Omega_{F}}$ is completely positive on 
$\Omega_{F}$ guarantees that $\Gamma$ is completely positive as a 
kernel on all of $\Omega$.  We have now completed the proof of (1) 
$\Rightarrow$ (2) in Theorem \ref{T:SchurAgler}.
\end{proof}

\begin{proof}[Proof of (2) $\Rightarrow$ (3)]
    We are given a completely positive kernel $\Gamma$ on $\Omega_{0}$ so that 
    \eqref{Aglerdecom} holds for $z,w \in \Omega_{0}$.  By condition 
    (3$^{\prime}$) in Theorem \ref{T:cpker}, $\Gamma$ has a 
    decomposition of the form
    $$
    \Gamma(z,w)[g] = H(z) \rho(g) H(w)^{*}
    $$
 where $H \colon \Omega_{0} \to \cL(\cX, \cY)$ for an auxiliary 
 Hilbert space $\cX$ which also carries a $*$-representation $\rho$ 
 of the $C^{*}$-algebra $C_{b}(\Psi, \cL(\cY_{T}))$.  
 From \eqref{Aglerdecom} we then deduce
 \begin{align*}
     I - S_{0}(z) S_{0}(w)^{*} & = \Gamma(z,w)[I - {\mathbb E}(z) 
     {\mathbb E}(w)^{*}] \\
     & = H(z) \rho(I - {\mathbb E}(z) {\mathbb E}(w)^{*}) H(w)^{*} \\
     & = H(z) H(w)^{*} - H(z) L_{{\mathbb E}(z)^{*}}^{*} L_{{\mathbb 
     E}(w)^{*}} H(w)^{*}
 \end{align*}
 where we use \eqref{tensorid}.  This in turn can be rearranged as
 $$
 H(z) L_{{\mathbb E}(z)^{*}}^{*} L_{{\mathbb 
     E}(w)^{*}} H(w)^{*} + I = H(z) H(w)^{*} + S_{0}(z) S_{0}(w)^{*}
 $$
 which leads to the inner product identity
 \begin{align*}
 &  \langle L_{{\mathbb E}(w)^{*}}H(w)^{*} y_{w}, L_{{\mathbb E}(z)^{*}} 
 H(z)^{*} y_{z} \rangle_{C_{b}(\Psi, \cL(\cY_{T}, \cU_{T}) \otimes 
 \cX} + \langle y_{w}, y_{z} \rangle_{\cY}  \\
 & \quad =
 \langle H(w)^{*} y_{w}, H(z)^{*} y_{z} x \rangle + \langle 
 S_{0}(w)^{*} y_{w}, S_{0}(z)^{*} y_{z} \rangle_{\cU}
 \end{align*}
 for arbitrary $y_{w}$ and $y_{z}$ in $\cY$.
 It then follows that the mapping $V$ given by
 \begin{equation}   \label{V}
  V \colon \begin{bmatrix} L_{{\mathbb E}(w)^{*}} H(w)^{*} y_{w} \\ 
  y_{w} \end{bmatrix} \mapsto \begin{bmatrix} H(w)^{*} y_{w} \\ 
  S_{0}(w)^{*} y_{w} \end{bmatrix}
 \end{equation}
 extends by linearity and continuity to a well-defined isometry from 
 the subspace
 $$
 \cD: = \overline{\operatorname{span}} \left\{ \begin{bmatrix} 
 L_{{\mathbb E}(w)^{*}} H(w)^{*} y_{w} \\ y_{w} \end{bmatrix} \colon 
 y_{w} \in \cY, \, w \in \Omega \right\} \subset
 \begin{bmatrix} C_{b}(\Psi, \cL(\cY_{T}, \cU_{T}))\otimes \cX) \\ 
     \cY \end{bmatrix}
 $$
 onto the subspace
 $$
 \cR: =  \overline{\operatorname{span}} \left\{ \begin{bmatrix} 
 H(w)^{*} y_{w} \\ S_{0}(w)^{*} y_{w} \end{bmatrix} \colon 
 y_{w} \in \cY, \, w \in \Omega \right\} \subset
 \begin{bmatrix}  \cX \\ \cU \end{bmatrix}.
 $$
 By replacing $\cX$ with $\cX' = \cX \oplus \widetilde \cX$ where 
 $\widetilde \cX$ is an infinite-dimensional Hilbert space if 
 necessary, we can arrange that the defect spaces $\left[ 
 \begin{smallmatrix} \cX' \\ \cY \end{smallmatrix} \right]  \ominus \cD$ 
 and $\left[ \begin{smallmatrix} \cX' \\ \cU \end{smallmatrix} \right] \ominus \cR$ have the same dimension.  
 We may also assume 
 that $\widetilde \cX$ is equipped with some representation $\widetilde 
 \rho$ of $C_{b}(\Psi, \cL(\cY_{T}))$ and hence $\cX'$ is equipped 
 with the representation $\rho' = \rho \oplus \widetilde \rho$.
 We now assume that all this has been done and drop the prime 
 notation; thus without loss of generality we have 
 $\operatorname{dim}\left[ \begin{smallmatrix} \cX  \\ \cY 
\end{smallmatrix} \right] \ominus \cD = \operatorname{dim} \left[ 
\begin{smallmatrix} \cX  \\ \cU \end{smallmatrix} \right] \ominus 
 \cR$ and $\cX$ is equipped with a $*$-representation $\rho$ of 
 $C_{b}(\Psi, \cL(\cY_{T}))$.  
 
 We now let $V_{0}$ be any unitary transformation from $\left[ 
 \begin{smallmatrix} C_{b}(\Psi, \cL(\cY_{T}, \cU_{T})) \otimes \cX  \\ \cY \end{smallmatrix} \right] \ominus 
     \cD$ onto $\left[ \begin{smallmatrix} \cX \\ \cU 
 \end{smallmatrix} \right] \ominus \cR$ and set 
 \begin{align*}
 \bU^{*}  = V \oplus V_{0}  \colon  &  \left[ \begin{smallmatrix}  C_{b}(\Psi, \cL(\cY_{T}, \cU_{T})) \otimes \cX \\  
 \cY\end{smallmatrix} \right] \cong \cD \oplus 
 \left(\left[ \begin{smallmatrix}  C_{b}(\Psi, \cL(\cY_{T}, \cU_{T})) \otimes\cX \\ \cY \end{smallmatrix} 
 \right]  \ominus \cD \right)   \\
 & \quad  \to
 \left[ \begin{smallmatrix} \cX \\ \cU \end{smallmatrix} \right] \cong  \cR \oplus 
\left( \left[ \begin{smallmatrix} \cX \\ \cU \end{smallmatrix} \right]
 \ominus \cR \right).
 \end{align*}
 We may then write out $\bU^{*}$ as a block $2 \times 2$-matrix
 $$
 \bU = \begin{bmatrix} A^{*} & C^{*} \\ B^{*} & D^{*} \end{bmatrix} 
 \colon \begin{bmatrix}  C_{b}(\Psi, \cL(\cY_{T}, \cU_{T})) \otimes \cX
 \\ \cY \end{bmatrix} \to \begin{bmatrix}  \cX \\ \cU \end{bmatrix}.
 $$
 Since $\bU^{*}$ is an extension of $V$ given by \eqref{V}, we have
 \begin{equation} \label{bUaction}
 \begin{bmatrix} A^{*} & C^{*} \\ B^{*} & D^{*} \end{bmatrix} 
     \begin{bmatrix} L_{{\mathbb E}(w)^{*}}H(w)^{*} y_{w} \\ 
	 y_{w}\end{bmatrix} = \begin{bmatrix} H(w)^{*} y_{w} \\ 
	 S_{0}(w)^{*} y_{w} \end{bmatrix}.
\end{equation}
The first row of \eqref{bUaction} gives
$$
A^{*} L_{{\mathbb E}(w)^{*}}H(w)^{*} y_{w} + C^{*} y_{w} = H(w)^{*} 
y_{w}.
$$
Since $\sup_{\psi} \{ \| \psi(w)\| \} < 1$ by the assumption 
\eqref{test-axiom} and since $\|A^{*}\| \le 1$ as $\bU$ is unitary, we 
see that $I - A^{*} L_{{\mathbb E}(w)^{*}}$ is invertible and, by
the arbitrariness of $y_{w} \in \cY$, we can 
solve \eqref{bUaction} to get
$$
  H(w)^{*} = (I - A^{*} L_{{\mathbb E}(w)^{*}})^{-1} C^{*}.
$$
Plugging this into the second row of \eqref{bUaction} then gives
$$
  B^{*} L_{{\mathbb E}(w)^{*}}(I - A^{*} L_{{\mathbb E}(w)^{*}})^{-1} 
  C^{*} + D^{* }= S_{0}(w)^{*}.
$$
Taking adjoints and replacing $w$ by $z \in \Omega_{0}$ leads to the 
realization formula \eqref{S0realization}.

We actually get a little bit more.  The right-hand side of 
\eqref{S0realization} makes sense for $z$ equal to any point 
in $\Omega$.  Thus we have actually proved: (2) $\Rightarrow$ 
(3$^{\prime}$) 
where the precise statement of (3$^{\prime}$) is:

\begin{enumerate}
    \item[(3$^{\prime}$)] {\em There is a $\Psi$-unitary colligation 
    $\bU$ as in \eqref{bUcol} such that $S_{0}$ has an  extension to 
    an $\cL(\cU, \cY)$-valued function $S$ defined on all of $\Omega$ 
    having the transfer-function realization
    \begin{equation}   \label{Srealization}
	S(z) = D + C (I - L^{*}_{{\mathbb E}(z)^{*}} A)^{-1} 
	L^{*}_{{\mathbb E}(z)^{*}} B
   \end{equation}
   for $z \in \Omega$.}
\end{enumerate}
 \end{proof}
 
 \begin{proof}[Proof of (3) $\Rightarrow$ (2)]  We assume that we 
     have  a transfer-function realization \eqref{S0realization} and 
     we must produce a completely positive kernel $\Gamma$ so that 
     \eqref{Aglerdecom} holds.  There is a natural candidate, namely:
     \begin{equation}   \label{candidate}
  \Gamma(z,w)[g] = C (I - L^{*}_{{\mathbb E}(z)^{*}}A)^{-1} \rho(g) 
  (I - A^{*} L_{{\mathbb E}(w)^{*}})^{-1} C^{*}.
  \end{equation}
  
  The candidate is certainly a completely positive kernel since the 
  formula \eqref{candidate} exhibits its Kolmogorov decomposition 
  (condition (3$^{\prime}$) in Theorem \ref{T:cpker} with $H(z) =  C 
  (I - L^{*}_{{\mathbb E}(z)^{*}}A)^{-1}$ and $\pi = \rho$).
  The verification of \eqref{Aglerdecom} amounts to the identity
  \begin{equation}  \label{toshow}
      I - S_{0}(z) S_{0}(w)^{*} = C (I - L^{*}_{{\mathbb 
      E}(z)^{*}}A)^{-1} \rho(I - {\mathbb E}(z) {\mathbb E}(w)^{*}) (I - 
      A^{*} L_{{\mathbb E}}(w)^{*})^{-1} C^{*}.
   \end{equation}
   Using the realization formula \eqref{S0realization} for 
   $S_{0}(z)$ and the relations 
   $$
   AA^{*} + BB^{*} = I, \quad AC^{*} + B D^{*} = 0, \quad CC^{*} + 
   DD^{*} = I
   $$
   coming out of the coisometric property $\bU \bU^{*} = I$ of $\bU$ 
   then give us
   \begin{align}
     &  I - S_{0}(z) S_{0}(w)^{*}  \notag \\
     & \quad = I - [D + C (I - L_{{\mathbb 
       E}(z)^{*}}^{*} A)^{-1}L^{*}_{{\mathbb E}(z)^{*}} B ] [D^{*} + B^{*} L_{{\mathbb 
       E}(w)^{*}} (I - A^{*} L_{{\mathbb E}(w)^{*}})^{-1} C^{*} ] 
       \notag \\
       & \quad  = I - D D^{*} - C (I - L_{{\mathbb 
       E}(z)^{*}}^{*} A)^{-1}L^{*}_{{\mathbb E}(z)^{*}} B D^{*} - 
       D  B^{*} L_{{\mathbb 
       E}(w)^{*}} (I - A^{*} L_{{\mathbb E}(w)^{*}})^{-1} C^{*}\notag  \\
       & \quad \quad 
       -  C (I - L_{{\mathbb 
       E}(z)^{*}}^{*} A)^{-1}L^{*}_{{\mathbb E}(z)^{*}} B B^{*} L_{{\mathbb 
       E}(w)^{*}} (I - A^{*} L_{{\mathbb E}(w)^{*}})^{-1} C^{*}\notag  \\
       & = C C^{*} + C (I - L_{{\mathbb E}(z)^{*}}^{*} A)^{-1}L^{*}_{{\mathbb E}(z)^{*}} A C^{*} 
       + C A^{*}  L_{{\mathbb E}(w)^{*}} (I - A^{*} L_{{\mathbb 
       E}(w)^{*}})^{-1} C^{*}  \notag \\
       & \quad \quad 
       +  C (I - L_{{\mathbb E}(z)^{*}}^{*} A)^{-1}L^{*}_{{\mathbb E}(z)^{*}} (AA^{*} - I)
       L_{{\mathbb E}(w)^{*}} (I - A^{*} L_{{\mathbb E}(w)^{*}})^{-1} 
       C^{*} \notag  \\
       & = C (I - L^{*}_{{\mathbb E}(z)^{*}} A)^{-1} X (I - A^{*} L_{{\mathbb E}(w)^{*}})^{-1} C^{*}
       \label{get1}
   \end{align}
   where we have set $X$ equal to
   \begin{align}
       X & = (I - L^{*}_{{\mathbb E}(z)^{*}} A) (I - A^{*}  L_{{\mathbb E}(w)^{*}}) 
       + L^{*}_{{\mathbb E}(z)^{*}} A (I - A^{*} L_{{\mathbb 
       E}(w)^{*}}) \notag  \\
       & \quad \quad + (I - L^{*}_{{\mathbb E}(z)^{*}} A) A^{*} L_{{\mathbb E}(w)^{*}}
      + L^{*}_{{\mathbb E}(z)^{*}} A A^{*} L_{{\mathbb E}(w)^{*}} - 
       L^{*}_{{\mathbb E}(z)^{*}} L_{{\mathbb E}(w)^{*}} \notag  \\
       & = I - L^{*}_{{\mathbb E}(z)^{*}} L_{{\mathbb E}(w)^{*}} 
       \notag \\
       & = \rho( I - {\mathbb E}(z) {\mathbb E}(w)^{*} )
       \label{get2}
   \end{align}
   where we used \eqref{tensorid} for the last step.  Combining 
   \eqref{get1} and \eqref{get2} gives us \eqref{toshow} as 
   required.
     \end{proof}
     
     \begin{proof}[Proof of (2) $\Rightarrow$ (1) if 
	 $\operatorname{dim} \cY_{T} < \infty$]
	 We assume that we have an Agler decomposition 
	 \eqref{Aglerdecom} and must show that $S_{0}$ can be 
	 extended to an $S$ defined on all of $\Omega$ which is in 
	 the Schur-Agler class $\mathcal{SA}_{\Psi}(\cU, \cY)$.  
	 Toward this end, we note that the proof of (2) $\Rightarrow$ 
	 (3) really proved (3$^{\prime}$), i.e., that $S_{0}$ extends 
	 to an $S$ defined on all of $\Omega$ given by the 
	 realization formula \eqref{Srealization}.  Therefore the 
	 argument behind (3) $\Rightarrow$ (2) actually gives us an 
	 Agler decomposition \eqref{Aglerdecom} valid for the 
	 extended $S$ which holds for $z,w$ in all of $\Omega$.  In 
	 this way we may assume that $S$ is given to us defined on 
	 all of $\Omega$ and we are given the completely positive 
	 kernel $\Gamma$ on all of $\Omega$ giving rise to the Agler 
	 decomposition \eqref{Aglerdecom} for $S$.
	 
	 To check that $S$ is in the Schur-Agler class 
	 $\mathcal{SA}_{\Psi} (\cU, \cY)$, we must verify that the 
	 operator $R_{S}$ of right multiplication by $S$ is 
	 contractive from $ \cH(K)_{\cY}$ to 
	 $\cH(K)_{\cU}$ for any choice of admissible 
	 kernel $K \in \cK_{\Psi}(\cY)$.  Toward this end, we reverse 
	 the procedure used in the proof of (1) $\Rightarrow$ (2) as 
	 follows. 
	 
	 Given an admissible kernel $K \in \cK_{\Psi}$ and given any 
	 finite collection of points $z_{1}, \dots, z_{N} \in 
	 \Omega$, we must show that the kernel \eqref{kYSK} is a  
	 positive kernel for all choices of functions $Y \colon 
	 \{z_{1}, \dots, z_{n} \} \to \cC_{2}(\cY, \cU)$.  It 
	 suffices to consider the restriction $K_{0}$ of $K$ to the 
	 finite set $\Omega_{0} = \{z_{1}, \dots, z_{N}\}$.  Since $K 
	 \in \cK_{\Psi}(\cY)$, we know that the right multiplication 
	 operator $R_{\psi}$ is contractive from $ \cH(K_{0})_{\cY_{T}}$ 
	 to $\cH(K_{0})_{\cU_{T}}$ 
	 for each $\psi \in \Psi$.  Consider the modified kernel
	 $$
	 K_{0,\epsilon}(z,w) = K_{0}(z,w) + \epsilon^{2} \sum_{z \in 
	 \Omega_{0}} \delta_{z,w} I_{\cY}
	 $$
	 where $\delta_{z,w}$ is the Kronecker delta function equal 
	 to $1$ for $z=w$ and $0$ otherwise.  Since the values of 
	 $\psi$ are contractive, we see that $R_{\psi}$ is still 
	 contractive as an operator from $\cH(K_{0,\epsilon})_{\cY_{T}}$ to  
	 $\cH(K_{0,\epsilon})_{\cU_{T}}$ for each $\epsilon > 
	 0$.  Also, to show that $R_{S}$ is contractive from 
	 $\cH(K_{0})_{\cY}$ to  $\cH(K_{0})_{\cU}$, it is enough to show that $R_{S}$ 
	 is contractive from $ \cH(K_{0,\epsilon})_{\cY}$ to  $\cH(K_{0,\epsilon})_{\cU}$ 
	 for each $\epsilon > 0$.
	 
	 Our next goal is to construct a kernel $L_{\epsilon} \colon 
	 \Omega_{0} \times \Omega_{0} \to \cL(\cY)$ so that
	 \begin{equation}   \label{Linnerprod}
	 \langle f, g \rangle_{\cH(K_{0,\epsilon})} = \sum_{z,w \in 
	 \Omega_{0}} \operatorname{tr}\left( L_{\epsilon}(z,w) f(z) g(w)^{*} 
	 \right).
	 \end{equation}
	 To do this, define $L(z,w) \in \cL(\cY)$ by
	 $$
	 \langle L_{\epsilon}(z,w) u,\, v \rangle_{\cY} = \langle \delta_{z} 
	 u,\, \delta_{w} v \rangle_{\cH(K_{0, \epsilon}}
	 $$
	 where $\delta_{z}$ is the point-mass function
	 $$
	  \delta_{z}(z') = \begin{cases} 1 & \text{ if } z = z', \\
	     0 & \text{ otherwise}.
	     \end{cases}
	 $$
	 In terms of the kernel function $K_{0,\epsilon}$, one can 
	 verify the block-matrix identity
	 $$
	[ L_{\epsilon}(z,w) ]_{z,w, \in \Omega_{0}} = \left( [K_{0, 
	\epsilon}(z,w) ]_{z,w \in \Omega_{0}} \right)^{-1}.
	$$
	The fact that  $R_{\psi} \colon \cH(K_{0, 
	\epsilon})_{\cY_{T}} \to \cH(K_{0,\epsilon})_{\cU_{T}}$ is contractive can be 
	equivalently expressed as
	\begin{equation}   \label{Rpsicontractive}
	\sum_{z,w, \in \Omega_{0}} \operatorname{tr} \left( 
	L_{\epsilon}(z,w) f(z) (I - \psi(z) \psi(w)^{*}) f(w)^{*} 
	\right) \ge 0 \text{ for all } f \colon \Omega \to 
	\cC_{2}(\cY_{T}, \cY).
	\end{equation}
	To show that $R_{S} \colon \cH(K_{0, \epsilon})_{\cY} \to 
	\cH(K_{0, \epsilon})_{\cU}$ is contractive can be expressed 
	in a similar way as
\begin{equation}   \label{RScontractive}
    \sum_{z,w \in \Omega_{0}} \operatorname{tr} \left( 
    L_{\epsilon}(z,w) h(z) (I - S(z) S(w)^{*}) h(w)^{*} \right) \ge 0 
    \text{ for all } h \colon \Omega_{0} \to \cC_{2}(\cY).
\end{equation}

By assumption we are given an Agler decomposition \eqref{Aglerdecom} 
for $S$.  The completely positive kernel $\Gamma$ appearing in 
\eqref{Aglerdecom} in turn has a Kolmogorov decomposition
as in (3$^{\prime}$) in Theorem \ref{T:cpker}:
\begin{equation}  \label{Krep}
\Gamma(z,w)[g] = H(z) \rho(g) H(w)^{*}
\end{equation}
for a $*$-representation $\rho \colon C_{b}(\Psi, \cL(\cY_{T})) \to 
\cL(\cX)$.  We now use the assumption that $\operatorname{dim} \cY_{T} 
< \infty$.  This has the effect that $C_{b}(\Psi, \cL(\cY_{T}))$ is a 
CCR $C^{*}$-algebra and that any representation $\rho$ of 
$C_{b}(\Psi, \cL(\cY_{T}))$ is the direct integral of multiples of 
irreducible representations, where an irreducible representation 
$\pi_{0} \colon C(\Psi_{\beta}, \cL(\cY_{T})) \to \cL(\cY_{T})$ has 
the point-evaluation form $\pi_{0}(g) = g(\psi_{0})$ for some $\psi_{0} \in \Psi_{\beta}$;   
we refer to \cite{Arv} and \cite[Section 2.3]{GH} for fuller discussion. 
Thus we may assume that there are mutually singular measures 
$\mu_{\infty}, \mu_{1}, \mu_{2}, \dots $ defined on the Borel subsets 
of the Stone-\v{C}ech compactification $\Psi_{\beta}$ of $\Psi$ so that
$$
 \rho = \infty \cdot \pi_{\mu_{\infty}} \oplus 1 \cdot \pi_{\mu_{1}} 
 \oplus 2 \cdot \pi_{\mu_{2}} \oplus \cdots
$$
where
$$
  \pi_{\mu_{j}}(g)  \colon f(\psi) \mapsto g(\psi) f(\psi)
  \text{ on  }\cH_{\pi_{j}} : =  L^{2}_{\cY_{T}}(\mu_{j}) = L^{2}(\mu_{j}) \otimes \cY_{T}
 $$
 and where in general $n \cdot \pi$ refers to the $n$-fold inflation 
 of $\pi$:
 $$
 (n \cdot \pi)(g) = \begin{bmatrix}  \pi(g) & & \\ & \ddots & \\ & & 
 \pi(g) \end{bmatrix} \text{ on } (\cH_{\pi})^{n} : = 
 \bigoplus_{j=1}^{n} \cH_{\pi}.
 $$
 Thus we may assume that the representation space $\cX$ in \eqref{Krep} 
 decomposes as
 $$
  \cX = L^{2}_{\cY_{T}}(\mu_{\infty})^{\infty} \oplus \bigoplus_{r=1}^{\infty} 
  L^{2}_{\cY_{T}}(\mu_{r})^{r}.
 $$
 Therefore the operators $H(w)^{*}$ appearing in \eqref{Krep} 
 decompose as
 $$
 H(w)^{*} = \begin{bmatrix} H_{\infty}(w)^{*} \\ 
 \operatorname{col}_{r=1}^{\infty} H_{r}(w)^{*} \end{bmatrix}
 $$
 where each $H_{r}(w)^{*}$ is an operator from $\cY$ to 
 $L^{2}_{\cY_{T}}(\mu_{r})^{r}$.  This enables us to define an operator-valued 
 function $H_{r}(w, \psi)^{*}$ of $\psi \in \Psi_{\beta}$ according to 
 $$
   H_{r}(w, \psi)^{*} y = \left((H_{r}(w)^{*} y \right)(\psi).
 $$
 Then the adjoint  $H_{r}(z)$ of $H_{r}(z)^{*}$ is given via an integral 
 formula:
 $$
 H_{r}(z) \colon G(\psi) \mapsto \int_{\Psi_{\beta}} H_{r}(z,  \psi) 
 G(\psi) {\tt d}\mu_{r}(\psi).
 $$
 We conclude that the Agler decomposition \eqref{Aglerdecom} takes 
 the more detailed form
 \begin{align}
  I - S(z) S(w)^{*} =  &  \int_{\Psi_{\beta}} H_{\infty}(z, \psi) \left( 
 I_{\ell^{2}} \otimes (I - \psi(z) \psi(w)^{*}) \right) H_{\infty}(w, 
 \psi)^{*} {\tt d}\mu_{\infty}(\psi)   \notag \\
 & + 
 \sum_{r=1}^{\infty}  \int_{\Psi_{\beta}} H_{r}(z, \psi) \left( 
 I_{{\mathbb C}^{r}} \otimes (I - \psi(z) \psi(w)^{*}) \right) H_{r}(w, 
 \psi)^{*} {\tt d}\mu_{r}(\psi).
 \label{detailed}
 \end{align}
 Plugging this into the left-hand side of the desired inequality in  \eqref{RScontractive}
 and taking the integral to the outside gives us the sum over $z,w 
 \in \Omega_{0}$ of the following terms:
 \begin{align*}
    &  \int_{\Psi_{\beta}} \operatorname{tr} \left( 
    L_{\epsilon}(z,w) h(z) H_{\infty}(z, \psi) \left( 
 I_{\ell^{2}} \otimes (I - \psi(z) \psi(w)^{*}) \right) H_{\infty}(w, 
 \psi)^{*}h(w)^{*}\right) {\tt d}\mu_{\infty}(\psi)  + \\
 &    \sum_{r=1}^{\infty} \int_{\Psi_{\beta}}
 \operatorname{tr} \left( L_{\epsilon}(z,w) h(z)
  H_{r}(z, \psi) \left(I_{{\mathbb C}^{r}} \otimes (I - \psi(z) 
  \psi(w)^{*}) \right)  H_{r}(w,\psi)^{*} h(w)^{*}  \right) 
  {\tt  d}\mu_{r}(\psi).
 \end{align*}
 From \eqref{Rpsicontractive} we see that the sum over $z,w \in 
 \Omega_{0}$ of the integrand in each of these terms is nonnegative.  
 Hence the sum over $z,w$ of the integrals in nonnegative and 
 \eqref{RScontractive} follows as required.
 \end{proof}
 
 \begin{remark}\label{R:interpolation} {\em The interpolation problem for the class 
     $\mathcal{SA}_{\Psi}(\cU, \cY)$ can be formulated as follows:  {\em Given 
     a subset $\Omega_{0}$ of $\Omega$ and a function $S_{0} \colon 
     \Omega_{0} \to \cL(\cU, \cY)$,  give necessary and sufficient 
     conditions for the existence of an $S \in 
     \mathcal{SA}_{\Psi}(\cU, \cY)$ such that $S|_{\Omega_{0}} = 
     S_{0}$.} Assuming that $\operatorname{dim} \cY_{T} < \infty$, 
     one gets a solution criterion (arguably not particularly 
     practical at this level of generality) immediately from 
     the equivalence (1) $\Leftrightarrow$ (2) in Theorem 
     \ref{T:SchurAgler} (where we use (2) in the more concrete form 
     \eqref{detailed}):  {\em the 
     $\mathcal{SA}_{\Psi}(\cU, \cY)$-interpolation problem has  a 
     solution if and only if there exists a matrix-valued function $(\psi, z) 
     \mapsto H_{\psi}(z)$ on $\Psi_{\beta} \times \Omega_{0}$, 
     bounded and  measurable in $\psi$ for each $z$, together with a 
     finite measure $\mu$ on $\Psi_{\beta}$, so that 
     $$
     I - S_{0}(z) S_{0}(w)^{*} = \int_{\Psi_{\beta}} H_{\psi}(z) 
     \left( I_{\cX_{\psi}} \otimes (I - \psi(z) \psi(w)^{*}) \right) 
     H_{\psi}(w)^{*}\, {\tt d}\mu(\psi)
     $$
     for each $z,w \in \Omega_{0}$.}  Not so  apparent from the way 
     Theorem \ref{T:SchurAgler} is formulated is that condition (1) 
     by itself is also a criterion for solving the interpolation 
     problem.  Indeed, if we set $\Psi|\Omega_{0}$ equal to the 
     collection of restricted functions 
     \begin{equation}   \label{Psirestricted}
     \Psi|\Omega_{0} = \{ \psi|_{\Omega_{0}} \colon \psi \in \Psi\},
     \end{equation}
     we may view $\Psi|\Omega_{0}$ as itself a collection of test 
     functions generating a Schur-Agler class 
     $\mathcal{SA}_{\Psi|\Omega_{0}}(\cU, \cY)$ of $\cL(\cU, 
     \cY)$-valued functions defined only on $\Omega_{0}$.  The only 
     part of the hypothesis that $S_{0}$ extends to an $S \in 
     \mathcal{SA}_{\Psi}$ used to prove (1) $\Rightarrow$ (2) in 
     Theorem \ref{T:SchurAgler} is that then 
     $S_{0} \in \mathcal{SA}_{\Psi|\Omega_{0}}$.  We conclude that we 
     get another criterion for solution of the interpolation 
     problem:  {\em the $\mathcal{SA}_{\Psi}(\cU, \cY)$-interpolation 
     problem has a solution if and only if $S_{0} \in 
     \mathcal{SA}_{\Psi|\Omega_{0}}$.}  Let us say that the subset 
     $\cK_{\Psi}^{0}(\cY)$ of the set of admissible kernels 
     $\cK_{\Psi}(\cY)$ is a {\em generating set} for $\cK_{\Psi}(\cY)$ 
     if, for each kernel $K \in \cK_{\Psi}(\cY)$, there is a 
     kernel $K^{0} \in \cK^{0}_{\Psi}(\cY)$ such that $K$ is 
     congruent to $K^{0}$ in the sense that there is an operator function 
     $Y$ so that $K(z,w) = Y(z) K^{0}(z,w) Y(w)^{*}$.  It is easy to 
     check that the kernels of the form \eqref{kYSK} are positive on 
     $\Omega_{0}$ for all $Y$ and admissible $K$ if and only if all such 
     kernels are positive when the admissible $K$ is restricted to 
     those coming from the generating set $\cK^{0}_{\Psi}(\cY)$.  
     Hence we arrive at the following dual criterion for solution of 
     the $\mathcal{SA}_{\Psi}(\cU, \cY)$-interpolation problem:  {\em 
     the $\mathcal{SA}_{\Psi}(\cU, \cY)$-interpolation problem has a 
     solution if and only if the kernel
     $$
      k(z,w) = \operatorname{tr}\left( Y(w)^{*} (I - S_{0}(w)^{*} 
      S_{0}(z)) Y(z) K^{0}(z,w) \right)
      $$
      is a positive kernel on $\Omega_{0}$ for all $Y \colon 
      \Omega_{0} \to \cC_{2}(\cY, \cU)$ for all admissible kernels $K$
      from the generating set $\cK^{0}_{\Psi}(\cY)$.}  We illustrate 
      these ideas on the examples discussed in Section \ref{S:ex} 
      below.
      This duality pairing between admissible kernels and test 
      functions is central to the operator-algebra point of view of Paulsen and 
      Solazzo  toward interpolation theory (see \cite{PaulsenIEOT, 
      PaulsenJFA, PS}).
   }\end{remark}
 
 There is also an operator-algebra point of view toward the 
 Schur-Agler class.  For convenience in the following discussion, we 
 take all the coefficient spaces $\cU$, $\cY$, $\cU_{T}$, and 
 $\cY_{T}$ to be the same space $\cU$ although this probably is not 
 essential. We abbreviate the notation $\mathcal{SA}_{\Psi}(\cU, 
 \cU)$ to $\mathcal{SA}_{\Psi}(\cU)$.  Let $\Psi|\Omega_{0}$ be as in 
 \eqref{Psirestricted} and let $H^{\infty}_{\Psi|\Omega_{0}}(\cU)$ denote the 
 space of all $\cL(\cU)$-valued functions $S_{0}$ on the subset 
 $\Omega_{0}$ of $\Omega$ such that there exists a positive $M < 
 \infty$ so that the kernel $k_{X,S_{0},K,M}$ given by  \eqref{kXSKM} is a positive kernel on 
 $\Omega_{0}$ for all choices of $X \colon \Omega_{0} \to 
 \cC_{2}(\cE, \cU)$ and for all choices of $K$  for which the kernel 
 $k_{Y,\psi,K,1}$ is positive for all choices of $Y \colon \Omega_{0} 
 \to \cC_{2}(\cU)$ and $\psi \in \Psi$, or, what is the same, such that the right 
 multiplication operator $R_{S}$ has norm at most $M$ as an operator 
 on $\cH(K)_{\cU}$  for all positive kernels $K$ for which $R_{\psi}$ 
 has norm at most 1 on $\cH(K)_{\cU}$ for all $\psi \in \Psi$.  We 
 define the $H^{\infty}_{\Psi|\Omega_{0}}$-norm 
 $\|S\|_{H^{\infty}_{\Psi|\Omega_{0}}}$ as the infimum of all such 
 positive numbers $M$.  Then $H^{\infty}_{\Psi|\Omega_{0}}(\cU)$ is 
 an operator algebra with unit ball equal to the Schur-Agler class 
 $\mathcal{SA}_{\Psi|\Omega_{0}}(\cU)$.  The following 
 representation-theoretic characterization of the Schur-Agler class 
 will be convenient in Section \ref{S:cR} below.
 
 \begin{theorem}  \label{T:representation}
     Suppose that $\Psi$, $\Omega_{0} \subset \Omega$, and $S_{0}$ 
     are as in Theorem \ref{T:SchurAgler} with $\cU = \cY = \cU_{T} = 
     \cY_{T}$.  In addition to conditions (1), (2), (3) in Theorem 
     \ref{T:SchurAgler}, consider:
     
     \begin{enumerate}
	 \item[(4)]  For any representation $\pi \colon 
	 H^{\infty}_{\Psi|\Omega_{0}}(\cU) \to \cL(\cK)$ such that 
	 $\| \pi(\psi)\| \le 1$ for all $\psi \in \Psi$, it also 
	 holds that $\| \pi(S_{0})\| \le 1$.  
   \end{enumerate}
   Then (4) $\Rightarrow$ (1).  If $\operatorname{dim}\, \cU < \infty$, 
   then also (2) $\Rightarrow $ (4) and (1), (2), (3), and (4) are 
   all equivalent.
   \end{theorem}
   
   \begin{proof}
   Assume (4) holds and suppose that $K \in 
   \cK_{\Psi|\Omega_{0}}(\cU)$ is an admissible kernel.  We now view 
   the map $\pi_{K} \colon H^{\infty}_{\Psi|\Omega_{0}}(\cU) \to 
   \cL(\cH(K)_{\cU})$ sending $G \in 
   H^{\infty}_{\Psi|\Omega_{0}}(\cU)$ to the right multiplication 
   operator $R_{G}$ on $\cH(K)_{\cU}$ as a representation 
   (technically, an anti-representation, but this does not affect the 
   final results).  By definition of $K \in 
   \cK_{\Psi|\Omega_{0}}(\cU)$, we have $\pi_{K}(\psi)\| \le 1$ for 
   each $\psi \in \Psi$.  Condition (4) then tells us that 
   $\pi(S_{0}) \| \le 1$, i.e., $R_{S_{0}}$ on $\cH(K)_{\cU}$ has 
   norm at most 1.  In this way we have verified condition (1).
   
   Conversely, we suppose $\operatorname{dim} \cY_{T} = 
   \operatorname{dim} \cU < \infty$ and that condition (2) holds.  As 
   in the proof of (2) $\Rightarrow$ (1) we see that (2) can be 
   written in the more explicit form \eqref{detailed}.  Given any 
   $\cL(\cU)$-valued kernel ${\mathbf K}(z,w)$ with a factorization 
   ${\mathbf K}(z,w) = F(z) G(w)^{*}$ with $F,G \in 
   H^{\infty}_{\Psi|\Omega_{0}}(\cU)$, we use the hereditary 
   functional calculus to extend a given representation $\pi$  of 
   $H^{\infty}_{\Psi|\Omega_{0}}(\cU)$ to such kernels according to 
   the rule
   $$
   \pi\left( F(z) G(w)^{*} \right) = \pi(F) \pi(G)^{*}.
   $$
   Applying $\pi$ to \eqref{detailed} (and using continuity to push 
   $\pi$ past the integral sign) gives
   \begin{align*}
      & I - \pi(S_{0}) \pi(S_{0})^{*}  =
       \int_{\Psi_{\beta}} \pi\left(H_{\infty}(\cdot, \psi) \right)
   \left(I_{\ell^{2}} \otimes (I - \pi(\psi) \pi(\psi)^{*} \right)
   \pi\left( H_{\infty}(\cdot, \psi) \right)^{*} \, {\tt 
   d}\mu_{\infty}(t) \\
   & \quad + \sum_{r=1}^{\infty} \int_{\Psi_{\beta}} \pi\left( H_{r}(\cdot, 
   \psi)\right) \left( I_{{\mathbb C}^{r}} \otimes (I - \pi(\psi) 
   \pi(\psi)^{*} ) \right) \pi\left( H_{r}(\cdot, \psi) \right)^{*} 
   \, {\tt d}\mu_{r}(\psi).
  \end{align*}
  From the fact that $\| \pi(\psi)\| \le 1$ for each $\psi \in \Psi$ 
  we read off from this last expression that $\| \pi(S_{0})\| \le 1$ 
  as well, i.e., (4) is verified.
  \end{proof}

 \begin{remark}  \label{R:history}{\em 
    In the proof of Theorem \ref{T:SchurAgler} we drew on a lot of 
    ideas which have been used in previous versions of this type of 
    result, starting with the seminal paper of Agler 
    \cite{Agler-Hellinger} and continuing with \cite{AMcC99, BT, AT, BTV, 
    EP, Tomerlin, BBQ, BGM, Ambrozie04, DMM07, DM07} as well as commutant 
    lifting versions \cite{BTV, BLTT, AE, McCS}.  In particular, the 
    cone separation argument in the proof of (1) $\Rightarrow$ (2) 
    and the proof of (2) $\Rightarrow$ (3) (the so-called 
    lurking-isometry argument) go back to \cite{Agler-Hellinger}. 
    However there are some new technical difficulties in the 
    test-function setting where some new ideas are required in order 
    to arrive at the final result; we now discuss some of these.
     
     In the proof of (1) 
     $\Rightarrow$ (2), the use of the $\epsilon^{2}$-perturbation 
     term in the definition of the $\cH_{{\mathbb L}_{1}, \epsilon}$ 
     norm is the ploy needed to make the point-evaluations $f \mapsto 
     f(w)$ bounded and enables us to avoid the hypothesis that the 
     set of test functions $\Psi$ separates the points of any finite 
     subset $\Omega_{F}$ of $\Omega$, as used in \cite{DMM07, DM07}.
     
     Our proof of (2) $\Rightarrow$ (1) (with the hypothesis that 
     $\operatorname{dim} \cY_{T} < \infty$) is close to the proof of 
     (3) $\Rightarrow$ (1) in \cite{DM07} (for the scalar-valued case)
     (which actually involves use 
     of the representation-theory formulation (4)).  These authors 
     make use of the spectral theorem for a representation of 
     $C_{b}(\Psi, {\mathbb C})$, approximating a general 
     representation $\rho$ by a ``simple representation'' 
     (approximation of the general integral in \eqref{detailed} by a 
     simple-function integrand).  Thus their proof also makes use of 
     the CCR character of $C_{b}(\Psi, {\mathbb C})$, and hence does 
     not appear to extend to the case $\operatorname{dim} \cY_{T} = 
     \infty$.
     }\end{remark}
 
 \section{Algebras arising from test functions}   \label{S:ex}
 
 In this section, rather than starting with a set of test functions 
 $\Psi$, we assume that we are given a function algebra $\cA$ and 
 then seek to determine a set of test functions $\Psi_{\cU, \cY}$ so 
 that the unit ball of the operator-valued version of $\cA$, say $\cA \otimes \cL(\cU, \cY)$ 
 where $\cU$, $\cY$ are two coefficient Hilbert spaces, can be 
 identified as the associated Schur-Agler class 
 $\mathcal{SA}_{\Psi_{\cU,\cY}}(\cU, \cY)$.  
 
 The classical example is the Hardy algebra over 
 the unit disk $\cA = H^{\infty}({\mathbb D})$. The operator-valued 
 version $\cA \otimes \cL(\cU, \cY)$ has unit ball equal to the 
 classical operator-valued Schur class $\cS(\cU, \cY)$, for which we 
 have the now  classical result:  $S \in \cS(\cU, \cY)$ if and only if the 
 associated de Branges-Rovnyak kernel $K_{S}(z,w) = [I - S(z) 
 S(w)^{*}]/(1 - z \overline{w})$ is a positive kernel on ${\mathbb 
 D}$.  If we let $K_{S}(z,w) = H(z) H(w)^{*}$ be the Kolmogorov 
 decomposition of $K_{S}$, then we arrive at
 $$
  I - S(z) S(w)^{*} = H(z) \left( (1 - z \overline{w}) I_{\cX}\right) 
  H(w)^{*}
 $$
 which is exactly the Agler decomposition \eqref{Aglerdecom} 
 corresponding to the singleton collection of test functions $\Psi = 
\{ \psi_{0}\}$ with $\psi_{0}$ equal to the coordinate function:  
$\psi_{0}(z) = z$.  For this case, moving from the scalar-valued case 
to the matrix- or operator-valued case necessitates no change in the 
choice of test-function set $\Psi$.  A similar  story holds for the 
case of the Schur-Agler class over the polydisk \cite{BT}, the 
Schur-multiplier class over the Drury-Arveson space \cite{BTV, EP}, and 
the Schur-Agler class over more general domains in ${\mathbb D}^{d}$ 
with matrix polynomial or analytic defining function \cite{BBQ, AE}.  
However the situation for the case where $\cA$ is the algebra of 
bounded analytic functions over a finitely connected planar domain 
$\cR$, or where $\cA$ is the constrained Hardy algebra over the unit 
disk (bounded holomorphic functions $f$ on ${\mathbb D}$ with the 
extra constraint that $f'(0) = 0$) is quite different.  We discuss 
each of these in turn.

 \subsection{The Schur class over a multiply connected planar 
 domain}   \label{S:cR}
 
 We let $\cR$ denote a bounded domain (connected, open set) in the 
 complex plane ${\mathbb C}$ whose boundary consists of $m+1$ smooth 
 Jordan curves $\partial_{0}$, $\partial_{1}$, $\dots$, $\partial_{m}$
 with $\partial_{0}$ denoting the boundary of the unbounded component 
 of the complement of $\cR$ in ${\mathbb C}$.  We let $\cS_{\cR}$ 
 denote the space of holomorphic functions mapping $\cR$ into the 
 unit disk, and $\cS_{\cR}(\cU, \cY)$ the operator-valued version 
 consisting of holomorphic functions on $\cR$ with values in the 
 closed unit ball $\overline{\cB} \cL(\cU, \cY)$ of bounded linear 
 operators between two coefficient Hilbert spaces $\cU$ and $\cY$.
 In \cite{DM05} there was identified a collection of inner functions
 $\{ s_{\bx} \colon \bx \in {\mathbb T}_{\cR}\}$,
 normalized to have value 1 at a fixed point $\zeta_{0} \in 
 \partial_{0}$ and to satisfy $s(t_{0}) = 0$ at a fixed point $t_{0} 
 \in \cR$, having exactly $m$ zeros in $\cR$ (the minimal 
 number possible for a single-valued inner function on $\cR$),
 and indexed by $\bx$ belonging to the $\cR$-torus ${\mathbb 
 T}_{\cR}: = \partial_{0} \times \partial_{1} \times \cdots \times 
 \partial_{m}$, so that any scalar Schur class function $s \in 
 \cS_{\cR}$ has an Agler decomposition \eqref{Aglerdecom} with 
 respect to the family $\Psi= \{ \psi_{\bx} \colon \bx \in {\mathbb 
 T}_{\cR}\}$ s in \eqref{scalarRAdecom} (or \eqref{detailed} 
 specialized to this case):
\begin{equation}   \label{scalarRAglerdecom} 1 - s(z) \overline{s(w)}
    = \int_{{\mathbb T}_{\cR}} h_{\bx}(z)\, \left(1 - 
 s_{\bx}(z) \overline{s_{\bx}(w)}\right) \, \overline{h_{\bx}(w)}\, {\tt 
 d}\nu(\bx).
 \end{equation}
 
 In more detail, the functions $s_{\bx}$ are constructed as follows.  
 Let $\bphi = \{ \phi_{1}, \dots, \phi_{m}\}$ be real-valued 
 continuous functions on $\partial \cR$ such that
 \begin{equation}  \label{phi}
     \{ \phi_{1}, \dots, \phi_{m}\} = \text{ basis for }
     L^{2}(\omega_{t_{0}}) \ominus [H^{2}(\omega_{t_{0}}) + 
     \overline{H^{2}(\omega_{t_{0}})}]
 \end{equation}
 where $\omega_{t_{0}}$ is the harmonic measure on $\partial \cR$ for some 
 fixed point $t_{0} \in \cR$ (so $h(t_{0}) = \int_{\partial \cR} 
 h(\zeta)\,{\tt d} \omega_{t_{0}}(\zeta)$ for $h$ harmonic on $\cR$ 
 and continuous on $\cR^{-}$), $H^{2}(\omega_{t_{0}})$ is the 
 associated Hardy space, and the overline indicates complex 
 conjugation---see e.g.~\cite{Fisher}.  Then given $\bx = (x_{0}, 
 x_{1}, \dots, x_{m}) \in {\mathbb T}_{\cR}$, there is a unique 
 choice of weights $w_{0}^{\bx}$, $w_{1}^{\bx}$, $\dots$, $w_{m}^{\bx}$, 
 each positive with sum equal to 1, so that
 \begin{equation}   \label{weights}
     \sum_{r-0}^{m} w_{r}^{\bx} \phi_{i}(x_{r}) = 0 \text{ for } i = 
     1, \dots, m
 \end{equation}
 (see \cite[Theorem 3.1.17]{AHR}).  Given any $\bx$ and the 
 associated weights $(w_{0}^{\bx,}  w_{1}^{\bx}, \dots, w_{m}^{\bx})$ 
 we associate the probability measure on $\partial \cR$:
 $$
 \mu_{\bx} : = \sum_{r=0}^{m} w_{r}^{\bx} \delta_{x_{r}}
 $$
 where $\delta_{x_{r}}$ is the unit point-mass measure at $x_{r}$.  
 The constraint \eqref{weights} guarantees that the harmonic function 
 $$
     h_{\bx}(z) = \int_{\partial \cR} \cP_{z}(\zeta)\, {\tt 
     d}\mu_{\bx}(\zeta)
 $$
 (where $\cP_{z}(\zeta)$ is the poisson kernel normalized to have 
 $\cP_{t_{0}}(\zeta) = 1$) has single-valued harmonic conjugate.  We 
 then define $f_{\bx}(z)$ to be the unique holomorphic function on 
 $\cR$ with
 $$
     \R f_{\bx}(z) = h_{\bx}(z) \text{ and } f_{\bx}(t_{0}) = 1.
 $$
 Finally we set
 \begin{equation}  \label{sbx}
     s_{\bx}(z) = \frac{f_{\bx}(z) - 1}{f_{\bx}(z) + 1}.
 \end{equation}
 Then $s_{\bx}$ are the inner functions appearing in 
 \eqref{scalarRAglerdecom}, apart from the additional normalization 
 that $s_{\bx}(\zeta_{0}) = 1$ at a fixed $\zeta_{0} \in 
 \partial_{0}$.  Then it is shown in \cite{DM07} that $\cS_{\cR} = 
 \mathcal{SA}_{\Psi_{\cR}}$ with the collection of test functions 
 $\Psi_{\cR}$ taken to be $\Psi_{\cR} = \{s_{\bx} \colon \bx \in {\mathbb T}_{\cR}\}$.
 There it is shown, at least for the annulus case ($m=1$), that, with 
 the additional normalization $s_{\bx}(\zeta_{0}) = 1$ imposed, that 
 $\Psi_{\cR}$ is minimal in the sense that no nonempty open subset of 
 ${\mathbb T}_{\cR}$ can be omitted and still have the decomposition 
 \eqref{scalarRAglerdecom} hold  for all $s \in \cS_{\cR}$.
 
 Before explaining the matrix generalization of \eqref{sbx}, we first recall 
 some ideas from \cite{BG}.  Suppose that we are given a collection
 $$
 \bphi = \left\{ \phi^{(1)} = \begin{bmatrix} \phi^{(1)}_{1} \\ 
 \vdots \\ \phi^{(1)}_{m} \end{bmatrix}, \dots, \phi^{(n)} = 
 \begin{bmatrix} \phi^{(n)}_{1} \\ \vdots \\ \phi^{(n)}_{m} 
 \end{bmatrix} \right\}
 $$
 of $n$ vectors in ${\mathbb R}^{m}$.  From $\bphi$ we form the block 
 column vectors
 $$
 \bphi \otimes I_{N} = \left\{ \phi^{(1)} \otimes I_{N}: = 
 \begin{bmatrix} \phi^{(1)}_{1} I_{N} \\ \vdots \\ \phi^{(1)}_{m} 
     I_{N} \end{bmatrix}, \dots, \phi^{(n)} \otimes I_{N} : = 
     \begin{bmatrix} \phi^{(n)}_{1} I_{N} \\ \vdots \\ \phi^{(n)}_{m} 
	 I_{N} \end{bmatrix} \right\}
$$
in $\left( {\mathbb C}^{N \times N}\right)^{m}$ ($m \times 1$-column 
vectors with entries of size $N \times N$).  We then say that the 
zero element ${\mathbf 0} = \left[ \begin{smallmatrix} 0_{N \times N} 
\\ \vdots \\ 0_{N \times N} \end{smallmatrix} \right]$ of 
$\left({\mathbb C}^{N \times N}\right)^{m}$ is {\em in the 
$C^{*}$-convex hull of $\bphi \otimes I_{N}$} if there exist positive 
semidefinite $N \times N$ matrices $W_{1}, \dots, W_{n}$ with 
$\sum_{r=1}^{n} W_{r} = I_{N}$ so that 
\begin{equation}   \label{0in}
    {\mathbf 0} = \sum_{r=1}^{n} \phi^{(r)} \otimes W_{r}
\end{equation}
where we set $\phi^{(r)} \otimes W_{r} =  \left[ \begin{smallmatrix} \phi^{(r)}_{1} 
W_{r} \\ \vdots \\ \phi^{(r)}_{m} W_{r} \end{smallmatrix} \right]$. 
We say that ${\mathbf 0}$ is {\em in the interior of the 
$C^{*}$-convex hull of $\bphi \otimes I_{N}$} if in addition the 
matrix weights $\{W_{1}, \dots, W_{n}\}$ have the property 
that their range spaces $\{ \operatorname{Ran} W_{1}, \dots, 
\operatorname{Ran} W_{n}\}$ are {\em $\bphi$-constrained weakly independent} by which 
we mean:  {\em whenever $T_{1}, \dots, 
  T_{n}$ are $N \times N$ complex Hermitian matrices with 
  $\operatorname{Ran} T_{r} \subset \operatorname{Ran} W_{r}$ for 
  each $r = 1, \dots, n$ such that
  $$
  \sum_{r=1}^{n} T_{r} = 0 \text{ and } \sum_{r=1}^{n} 
  \phi_{i}(x_{r}) T_{r} = 0 \text{ for } i=1, \dots, n,
  $$
  it follows that $T_{r} = 0$ for each $r = 1, \dots, n$.}
  When all this happens, we refer to $\{W_{1}, \dots, W_{n}\}$ as 
  {\em a choice of matrix barycentric coordinates of ${\mathbf 0}$ 
  with respect to $\bphi$.}  
  
  By way of motivation for these notions, note that, in case $N=1$ 
  and all the weights $W_{1} = w_{1}$, $\dots$, $W_{n} = w_{n}$ (now 
  complex numbers) are nonzero (which can be arranged simply by 
  discarding appropriate vectors $\phi^{(r)}$ from the list of 
  vectors $\bphi$), then ${\mathbf 0} = 0 \in {\mathbb R}^{m}$ in the 
  interior of the $C^{*}$-convex hull of $\bphi \otimes I_{1} = \bphi$ simply means 
  that the vector $0 \in {\mathbb R}^{m}$ is in the interior of the 
  simplex generated by the vectors $\phi^{(1)}, \dots, \phi^{(n)}$ 
  and that $w_{1}, \dots, w_{n}$ are the classical barycentric 
  coordinates for $0$ with respect to the simplex vertices 
  $\phi^{(1)}, \dots, \phi^{(m)}$.
  
  We are now ready to explain the matrix analogue of the $\cR$-torus 
  ${\mathbb T}_{\cR}$ used to parametrize the set of scalar test 
  functions \eqref{sbx}.  We define the matrix $\cR$-torus
 ${\mathbb T}_{\cR}^{N}$ to 
   consist of all pairs $(\bx, \bw)$ of the form $(\bx, \bw) = (x_{1}, \dots, x_{n}; W_{1}, 
   \dots, W_{n})$ where $x_{1}, \dots, x_{n}$ is a set of $n$ 
   distinct points in $\partial \cR$ such that ${\mathbf 0}$ is in 
   the interior of the $C^{*}$-convex hull of the set of vectors
   $\bphi(\bx) \otimes I_{N}$, where we set
   \begin{equation}   \label{bphix}
   \bphi(\bx) = \left\{ \phi(x_{1}) = \begin{bmatrix} \phi_{1}(x_{1}) 
   \\ \vdots \\ \phi_{m}(x_{1}) \end{bmatrix}, \dots, \phi(x_{n}) = \begin{bmatrix} 
   \phi_{1}(x_{n}) \\ \vdots \\ \phi_{m}(x_{n}) \end{bmatrix} 
   \right\},
   \end{equation}
   with $\phi_{1}, \dots, \phi_{m}$ as in \eqref{phi}, and with $\{W_{1}, 
   \dots, W_{n}\}$ is a choice of matrix barycentric coordinates for 
   ${\mathbf 0}$ with respect to $\bphi(\bx) \otimes I_{N}$.  In 
   particular, the condition \eqref{0in} in the present context 
   specializes to
   \begin{equation}   \label{matrixweights}
       \sum_{r=1}^{n} \phi_{i}(x_{r}) W_{r} = 0 \text{ for } i = 1, 
       \dots, m.
   \end{equation}

   For the case $N=1$, necessarily 
   $n=m+1$, after a reindexing the collection of points $(x_{0}, 
   x_{1}, \dots, x_{m})$ necessarily consists of exactly one point 
   from each boundary component $\partial_{0}, \dots, \partial_{m}$, 
   and the associated scalar weights $w_{0}^{\bx}, w_{1}^{\bx}, 
   \dots, w_{m}^{\bx}$ are uniquely determined by $\bx$.  For $N>1$, 
   the characterization of ${\mathbb T}^{N}_{\cR}$ is not so 
   explicit; nevertheless it is nonempty and is a well-defined 
   metrizable topological space which is in one-to-one correspondence 
   with a collection of quantum measures (positive matrix measures  
   with total mass equal to the identity matrix $I_{N}$) which we define next.
   For additional information we refer to \cite{BG}.
   
   Given $(\bx, \bw) \in {\mathbb T}^{N}_{\cR}$, we associate a 
   quantum measure $\mu_{\bx, \bw}$ by
   \begin{equation}  \label{matrixmeasure}
       \mu_{\bx, \bw} = \sum_{r=1}^{n} W_{r} \delta_{x_{r}} \text{ if }
       (\bx, \bw) = (x_{1}, \dots, x_{n}; \, W_{1}, \dots, W_{n}) \in 
       {\mathbb T}^{N}_{\cR}.
       \end{equation}
    Then a consequence of \eqref{matrixweights} is that the 
    matrix-valued harmonic function
    $$ H_{\bx, \bw}(z) = \int_{\partial \cR} \cP_{z}(\zeta) \, {\tt 
    d}\mu_{\bx, \bw}(\zeta)
    $$
    has a single-valued (matrix-valued) harmonic conjugate, and hence 
    there is a unique\-ly determined holomorphic function $F_{\bx, \bw}$ 
    on $\cR$ with
    $$
    \R F_{\bx, \bw}(z) = H_{\bx, \bw}(z) \text{ and } 
    F_{\bx,\bw}(t_{0}) = I_{N}.
  $$
  It can be shown that the collection of functions
  \begin{equation}  \label{HerglotzRextreme}
  \{ F_{\bx, \bw} \colon (\bx, \bw) \in {\mathbb T}^{N}_{\cR}\}
  \end{equation}
  is exactly the set of extreme points for the compact convex set 
  $\cH^{N}(\cR)_{I}$ of 
  normalized Herglotz functions  over $\cR$ given by
  $$
    \cH^{N}(\cR)_{I} = \{ F\colon \cR \mapsto {\mathbb C}^{N \times N} 
    \colon F \text{ holomorphic, } \R F(z) \ge 0 \text{ for } z \in 
    \cR, \, F(t_{0}) = I_{N} \}.
    $$
  Finally, we set
  \begin{equation}   \label{Sbxbw}
      S_{\bx, \bw}(z) = (F_{\bx,\bw}(z) + I)^{-1} (F_{\bx, \bw}(z) -I).
 \end{equation}
 Note that each $S_{\bx, \bw}(z)$ is an $N \times N$ matrix inner 
 function on $\cR$ normalized to satisfy $S(t_{0}) = 0$.  Then in 
 \cite{BG} it is shown that any matrix-valued function $S$ in the 
 Schur class $\cS_{\cR}({\mathbb C}^{N}, {\mathbb C}^{N})$ has an 
 Agler decomposition of the form
 \begin{equation}   \label{matrixRAglerdecom}
     I - S(z) S(w)^{*} = \int_{{\mathbb T}^{N}_{\cR}} H_{\bx, \bw}(z) 
     \left( I - S_{\bx, \bw}(z) S_{\bx, \bw}(w)^{*} \right) H_{\bx, 
     \bw}(w)^{*} \, {\tt d}\nu(\bx, \bw)
  \end{equation}
  for appropriate matrix functions $H_{\bx, \bw}(z)$ and probability 
  measure $\nu$ on ${\mathbb T}^{N}_{\cR}$.
  
  Following the arguments in \cite{DM07} (adapted to the matrix-valued 
  setting) leads to the following identification of the matrix Schur 
  class $\cS_{\cR}({\mathbb C}^{N})$ with a 
  matrix-valued test-function Schur class $\mathcal{SA}_{\Psi^{N}_{\cR}}$; 
  the main ingredients of the proof also appear in the more involved 
  proof of Theorem \ref{T:annulustestfunc} below.

  \begin{theorem}  \label{T:matrixSchurR}
      Let $\Psi^{N}_{\cR}$ be the collection of matrix inner functions
   \begin{equation}   \label{PsiN}
       \Psi^{N}_{\cR} = \{ S_{\bx, \bw} \colon (\bx, \bw) \in {\mathbb 
       T}^{N}_{\cR} \}
   \end{equation}
   with $S_{\bx, \bw}$ as in \eqref{Sbxbw}, with the additional 
   normalization $S_{\bx, \bw}(\zeta_{0}) = I_{N}$ at some fixed point 
   $\zeta_{0} \in \partial_{0}$.  Then the matrix-valued Schur class 
   $\cS_{\cR}({\mathbb C}^{N})$ is identical to the 
   matrix-valued test-function Schur-Agler class 
   $\mathcal{SA}_{\Psi^{N}_{\cR}}$
   associated with the collection of test functions $\Psi^{N}_{\cR}$ (as 
   defined by \eqref{kXpsiK} and \eqref{kYSK}).
   \end{theorem}
   
   Combined with Theorem \ref{T:SchurAgler} and Remark  
   \ref{R:interpolation}, we arrive at the following dual 
   formulations of interpolation criteria for the Nevanlinna-Pick 
   interpolation problem for the matrix Schur class over $\cR$.
   Before stating the result we need a little more background 
   concerning function theory on $\cR$.   There is a standard 
   procedure (see e.g.\cite{Abrahamse}) for introducing $m$ disjoint 
   simple curves $\gamma_{1}, \dots, \gamma_{m}$ so that $\cR 
   \setminus {\boldsymbol \gamma}$ (where we set ${\boldsymbol 
   \gamma}$ equal to the union ${\boldsymbol \gamma} = \gamma_{1} 
   \cup  \cdots \cup \gamma_{m}$) is simply connected.  For each cut 
   $\gamma_{r}$ we assign some orientation, so that points $z$ not on 
   $\gamma_{r}$ but in a sufficiently small neighborhood of 
   $\gamma_{r}$ in $\cR$ can be assigned a location of either ``to 
   the left'' or ``to the right''. For $f$ a vector-valued function on 
   $\cR$ and $z$ a point on some $\gamma_{r}$, we let $f(z_{+})$ 
   denote the limit of $f(\zeta)$ as $\zeta$ approaches $z$ from the
   right of $\gamma_{r}$ in $\cR$, and similarly, $f(z_{-})$ the limit 
   of $f(\zeta)$ as $\zeta$ approaches $z$ from the
   left of $\gamma_{r}$ in $\cR$, whenever these 
   limits exist.  Given a $\bU = (U_{1}, \dots, U_{m})$ in 
   $\cU(N)^{m}$ ($m$-tuples of unitary $N \times N$ matrices), we define a Hardy space 
   $H^{2}(\bU)$ to consist of functions $f \colon \cR \to {\mathbb 
   C}^{N}$, holomorphic on $\cR \setminus {\boldsymbol \gamma}$, 
   subject to the jump conditions $f(z_{-}) = U_{r} f(z_{+})$ for $z 
   \in \gamma_{r}$ for each $r = 1, \dots, m$ (so $\|f(z)\|^{2}$ is 
   continuous and single-valued on $\cR$), and so that the 
   well-defined integral 
   $$
   \|f\|^{2}_{H^{2}(\bU)}  = \int_{\partial \cR} \| f(\zeta) \|^{2} \, 
   {\tt d}\omega_{t_{0}}
   $$
   is finite.  Then the space $H^{2}(\bU)$ is a reproducing kernel 
   Hilbert space over $\cR$ (with some appropriate convention as to how elements 
   are defined on ${\boldsymbol \gamma}$); we denote its ${\mathbb 
   C}^{N \times N}$-valued reproducing kernel  function by
   $K^{\bU}$:   $H^{2}(\bU) = \cH(K^{\bU})$.  These kernels  enter 
   into the admissible-kernel formulation of the criterion for the 
   $\cS_{\cR}({\mathbb C}^{N})$-interpolation problem to have a 
   solution.
   
   \begin{theorem} \label{T:Rinterpolation}
       Suppose that we are given an $N \times N$ 
       matrix-valued function $S_{0}$ on the subset $\cR_{0}$ of 
       $\cR$.  Then the following are equivalent:
       \begin{enumerate}
	   \item There is a function $S$ in the Schur class 
	   $\cS_{\cR}({\mathbb C}^{N})$ with $S|_{\cR_{0}} = S_{0}$.
	   
	   \item There is a matrix-valued function $((\bx, \bw), z) 
	   \mapsto H_{\bx,\bw}(z)$ on ${\mathbb T}^{N}_{\cR} \times 
	   \cR_{0}$, bound\-ed and measurable in $(\bx, \bw)$ for each $z 
	   \in \cR_{0}$, together with a finite measure $\nu$ on 
	   ${\mathbb T}^{N}_{\cR}$ so that
	   $$
	   I - S_{0}(z) S_{0}(w)^{*} = \int_{{\mathbb T}^{N}_{\cR}} 
	   H_{\bx, \bw}(z) \left( I_{\cX_{\bx, \bw}} \otimes (I - 
	   S_{\bx, \bw}(z) S_{\bx, \bw}(w)^{*} \right) H_{\bx, 
	   \bw}(w)^{*} \, {\tt d}\nu(\bx, \bw)
	   $$
	   for all $z,w \in \cR_{0}$.
	   
	   \item For each  $\bU = (U_{1}, \dots, U_{m})$ in 
	   $\cU(N)^{m}$ and for each $Y \colon 
	   \cR_{0} \to {\mathbb C}^{N \times N}$, the kernel
\begin{equation}  \label{R-Pick}
k(z,w): = \operatorname{tr}\left( Y(w)^{*}( I - S_{0}(w)^{*} S_{0}(z) 
)Y(z) ) K^{\bU}(z,w) \right)
\end{equation}
is a positive kernel on $\cR_{0}$.
\end{enumerate}
   \end{theorem}
   
   \begin{proof}  The equivalence of (1) and (2) is a consequence of 
       Theorem \ref{T:SchurAgler}, once the result of Theorem 
       \ref{T:matrixSchurR} is plugged in.  
       
       The equivalence of (1) and (3) is a consequence of Remark 
       \ref{R:interpolation}, once it is verified that the set 
       \begin{equation}   \label{KU}
   (\Psi^{N}_{\cR})^{0}: = \{ K^{\bU} \colon \bU \in \cU(N)^{m}  \}
   \end{equation}
   is a generating set for the set of admissible kernels 
       $\cK_{\Psi^{N}_{\cR}}({\mathbb C}^{N})$.  Rather than doing 
       this, we observe that a solution criterion  for the 
       $\mathcal{SA}_{\cR}({\mathbb C}^{N})$-interpolation problem 
       was obtained in \cite[Theorem 1.5]{BBtH} (as a consequence of the lifting 
       theorem from \cite{Ball79}), but in a somewhat different, more 
       convoluted form than the form \eqref{R-Pick}.  If one works with right multiplication   
       operators on the space $\cH(K^{U})_{{\mathbb C}^{N}}$ 
       rather than with left multiplication operators on a left-side 
       tensoring of the reproducing kernel Hilbert space consisting 
       of row-vector functions as is done in \cite{BBtH}, one arrives 
       at the solution criterion \eqref{R-Pick} as presented here.
      \end{proof}
      
      \begin{remark}  \label{R:Rinterpolation}  {\em  We note that the 
	  scalar-valued case $N=1$ 
	  of criterion (3) in Theorem \ref{T:Rinterpolation} is due 
	  to Abrahamse \cite{Abrahamse}---note that the extra 
	  parameter $Y(z)$ washes out in this case.  It was later shown by 
	  Ball-Clancey \cite{BC} that no open subset of the kernels $K^{U}$ ($U \in 
	  \cU(1)^{m}$) can be omitted for the validity of this 
	  result.  However, for the case of the annulus, if one 
	  takes the set of interpolation nodes $\cR_{0}$ to be finite 
	  and prespecified, then two kernels suffice \cite{VF}. 
	  While the Abrahamse result extends to the matrix-valued 
	  setting for the annulus case (using only scalar-valued 
	  kernels), McCullough and Paulsen  \cite{McC01, McCP}, using 
	  the $C^{*}$-algebra approach to interpolation theory, 
	  showed that the Fedorov-Vinnikov result fails for the 
	  matrix-valued case.  All this story is reviewed nicely in 
	  \cite{DPRS}.  We do not address such minimality issues here.
	  }\end{remark}
   
   For the case of the annulus ($m=1$), by using  results of 
   McCullough \cite{McC95} it is possible to obtain a more explicit 
   test-function collection as follows.  We take $\cR$ to have the 
   concrete form $\cR = {\mathbb A}_{q}$ where
   $$
   {\mathbb A}_{q} = \{ z \in {\mathbb C} \colon q < |z| < 1\}
   $$
   for a number $q$ satisfying $0 < q < 1$.  It is established in 
   \cite{McC95} that there is a curve $t \mapsto \varphi_{t}$ of 
   inner functions on ${\mathbb A}_{q}$ (constructed from the Ahlfors 
   function for ${\mathbb A}_{q}$ based at the point $\sqrt{q} \in 
   {\mathbb A}_{q}$) with the following property:    for a 
   $(U,t) \in \cU(N) \times {\mathbb T}^{n}$ (where $\cU(n)$ denotes 
   the set of $N \times N$ unitary matrices and ${\mathbb T}^{n}$ is 
   the $N$-torus $\{t = (t_{1}, \dots, t_{n}) \colon |t_{j}| =1 
   \text{ for } 1 \le j \le N\}$),
   set 
   $$
   \Phi_{U,t}(z) = U \begin{bmatrix} \varphi_{t_{1}}(z) & & \\ & \ddots & 
   \\ & & \varphi_{t_{N}}(z) \end{bmatrix} 
   $$
   and
   $$
   R_{U,t}(z) = (I_{N} +  \Phi_{U,t}(z)) (I - \Phi_{U,t}(z))^{-1};
   $$
   {\em then, 
   for each $(\bx, \bw) \in {\mathbb T}^{N}_{{\mathbb A}_{q}}$ there 
   is a choice of invertible $N \times N$ matrix $X$ and a $(U,t) 
   \in \cU(N) \times {\mathbb T}^{N}$ so that
   \begin{equation} \label{McCresult}
       F_{\bx, \bw}(z) +  F_{\bx, \bw}(z)^{*}  = X \left( R_{U,t}(z)  + 
       R_{U,t}(z)^{*} \right) X^{*} \text{ for all } z \in {\mathbb 
       A}_{q}.
   \end{equation}
   } We are now ready to introduce a new test-function class for 
   $\cS_{{\mathbb A}_{q}}^{N}$, namely:
   \begin{equation}   \label{tildePhiN}
   \widetilde \Psi^{N}_{{\mathbb A}_{q}} = \{  \Phi_{U,t} \colon (U, t) \in \cU(N) 
   \times {\mathbb T}^{N} \}.
   \end{equation}
   We then have the following result.
   
   \begin{theorem}  \label{T:annulustestfunc}
       The matrix-valued Schur class over the annulus $\cS_{{\mathbb 
       A}_{q}}({\mathbb C}^{N})$ is identical to the 
       matrix-valued test-function Schur-Agler class 
       $\mathcal{SA}_{\widetilde \Psi^{N}_{{\mathbb A}_{q}}}$ where $\widetilde 
       \Psi^{N}_{{\mathbb A}_{q}}$ is given by \eqref{tildePhiN}.
   \end{theorem}
   
   \begin{proof}  Suppose first that $S \in \mathcal{SA}_{\widetilde 
       \Psi^{N}_{{\mathbb A}_{q}}}$.  Then the right multiplication operator $R_{S}$ is 
       contractive on $\cH(K)_{{\mathbb C}^{N}}$ for each admissible 
       kernel $K$ in $\cK_{\widetilde \Psi^{N}_{{\mathbb A}_{q}}}({\mathbb C}^{N})$.  
       Such kernels include the Fay kernel associated with the Hardy 
       space $H^{2}(\omega_{t}) \otimes {\mathbb C}^{N}$ over 
       ${\mathbb A}_{q}$.  This 
       observation is enough to conclude that $S \in 
       \cS_{{\mathbb A}_{q}}({\mathbb C}^{N})$.
       
       Conversely suppose that $S \in  \cS_{{\mathbb A}_{q}}({\mathbb 
       C}^{N})$.
       To show that $S \in \mathcal{SA}_{\widetilde 
       \Psi^{N}_{{\mathbb A}_{q}}}({\mathbb C}^{n})$, by 
       Theorem \ref{T:representation} it suffices to show:  {\em for any 
       representation $\pi \colon H^{\infty}_{\widetilde 
       \Psi^{N}_{{\mathbb A}_{q}}}({\mathbb C}^{N}) \to \cL(\cK)$ such that $\|\pi(\Phi_{U,t})\| \le 
       1$ for all $(U, t) \in \cU(N) \times {\mathbb T}^{N}$, it 
       follows that $\| \pi(S) \| \le 1$}.  By replacing $\pi$ with $r 
       \cdot \pi$ with $r < 1$ and then taking a limit as $r$ tends 
       to $1$, without loss of generality we may suppose that $\| 
       \pi(\Phi_{U,t} )\| < 1$ for each $(U,t)$.   Then $\pi(R_{U,t}) 
       = (I - \pi(\Phi_{U,t}))^{-1} (I + \pi(\Phi_{U,t}))$ is a 
       well-defined bounded operator on $\cK$ such that
       \begin{equation}  \label{RUTposrealpart}
       \pi(R_{U,t})  + \pi(R_{U,t}) ^{*} = 
    2 \left(I - \pi(\Phi_{U,t}) \right)^{-1}
    \left( I -  \pi(\Phi_{U,t}) \pi(\Phi_{U,t})^{*} \right) 
    \left( I - \pi(\Phi_{U,t})^{*}\right)^{-1} > 0.
    \end{equation}
       From 
       \eqref{McCresult}, we see that, for each fixed $(\bx, \bw) \in 
       {\mathbb T}^{N}_{{\mathbb A}_{q}}$,  $\pi\left( F_{\bx,\bw}\right)$ 
       is a well-defined bounded operator on $\cK$ satisfying 
       \begin{equation} \label{Fxwposreal}
       \pi\left( F_{\bx, \bw}\right) + \pi\left( F_{\bx, 
       \bw}\right)^{*} = 
       \left(X \otimes I_{\cK}\right) \left( \pi(R_{U,t}) + 
       \pi(R_{U,t})^{*}\right) \left( X^{*} \otimes I_{\cK} \right).
       \end{equation}
       From \eqref{RUTposrealpart} we read off that  $\pi\left( 
       F_{\bx, \bw}\right)$
       has positive real part.  We next obtain $\pi(S_{\bx, \bw})$ as 
       a Cayley transform of $\pi\left(F_{\bx, \bw}\right)$:
       $$
       \pi(S_{\bx, \bw}) = \left( \pi(F_{\bx, \bw}) + I\right)^{-1} 
       \left(\pi(F_{\bx, \bw}) - I \right).
       $$
       From the relation
       $$
       I - \pi(S_{\bx, \bw}) \pi(S_{\bx,\bw})^{*} = 2 \left( 
       \pi(F_{\bx, \bw}) + I \right)^{-1} 
       \left( \pi(F_{\bx, \bw}) + \pi(F_{\bx, \bw})^{*} \right)
       \left( \pi(F_{\bx, \bw})^{*} + I\right)^{-1}
       $$
       combined with \eqref{Fxwposreal}, we see that $\| \pi(S_{\bx, 
       \bw}) \| \le 1$.  Finally, since $S \in \cS_{{\mathbb 
       A}_{q}}({\mathbb C}^{N})$, $S$ has an Agler decomposition as 
       in \eqref{matrixRAglerdecom}.  Applying the hereditary 
       functional calculus with the representation $\pi$ through this 
       integral representation gives
       $$
       I - \pi(S) \pi(S)^{*} =
       \int_{{\mathbb T}^{N}_{\cR}} \pi(H_{\bx, \bw}) (I - 
       \pi(S_{\bx, \bw}) \pi(S_{\bx, \bw})^{*} ) \pi(H_{\bx, 
       \bw})^{*} \,{\tt d}\nu(\bx, \bw).
       $$
       Since $\|\pi(S_{\bx, \bw})\| \le 1$ for each $(\bx, \bw) \in 
       {\mathbb T}^{N}_{{\mathbb A}_{q}}$, we read off from this last 
       expression that $\| \pi(S) \| \le 1$
        \end{proof}

	    As a corollary of Theorem \ref{T:annulustestfunc} 
	    combined with Theorem \ref{T:SchurAgler}, we get the 
	    following structure theorem for the Schur-Agler class 
	    over the annulus ${\mathbb A}_{q}$.  To this end we 
	    introduce the space $\widehat {\mathbb T}^{N}_{{\mathbb 
	    A}_{q}} = ( \cU(N)/\cU(1)^{N}) \times {\mathbb T}^{N}$, 
	    where here $\cU(1)^{N}$ is identified with unitary diagonal 
	    $N \times N$ matrices, and the action of $\cU(1)^{N}$ on 
	    $\cU(N)$ is given by
	    $$
	    u  \colon U \mapsto U u \text{ for } u =
	     \begin{bmatrix} u_{1} & & \\ & \ddots & \\ & & u_{N} 
	     \end{bmatrix} \in \cU(1)^{N}.
	     $$
	    For $([U], t) \in \widetilde {\mathbb T}^{N}_{{\mathbb 
	    A}_{q}}$, we abuse notation somewhat and set 
	    $$
	    \Phi_{[U],t} =\Phi_{U,t} U^{*}.
	    $$
	    Note 
	    that $\Phi_{[U],t}$ is well-defined (independent of the 
	    choice of representative of the coset $[U]$).   Note that 
	    each $\Phi_{[U],t}$ is normalized to satisfy 
	    $\Phi_{[U],t}(1) = I_{N}$ as well as 
	    $\Phi_{[U],t}(\sqrt{q}) = 0$.  Furthermore the expression 
	    $I - \Phi_{U,t}(z) \Phi_{U,t}(w)^{*}$ is independent of 
	    choice of coset representative for $[U]$.
	    Also it is easily checked that the set of admissible 
	    kernels $\cK_{\Psi}$ associated with a given collection 
	    of test functions $\Psi$ depends on the functions $\psi 
	    \in \Psi$ only through the expressions 
	    $I - \psi(z) \psi(w)^{*}$.  Hence the result of Theorem 
	    \ref{T:annulustestfunc} can equally well be stated as:
	    \begin{equation}   \label{annulus:reduced}
		\cS^{N}_{{\mathbb A}_{q}}({\mathbb C}^{n}) = 
		\mathcal{SA}_{\widehat \Psi^{N}_{{\mathbb 
		A}_{q}}}({\mathbb C}^{N})
	\end{equation}
	where we have set
	$$\widehat \Psi^{N}_{{\mathbb A}_{q}} = \{ \Phi_{[U],t} 
	\colon ([U],t) \in \widehat {\mathbb T}^{N}_{{\mathbb A}_{q}}	\}.
	$$
	Then the following corollary is an immediate consequence of 
	our man theorem on the test-function Schur-Agler class, 
	namely Theorem \ref{T:SchurAgler}.

	    \begin{corollary} \label{C:annulus}
		Suppose that $S \in \cS_{{\mathbb A}_{q}}({\mathbb 
		C}^{N})$.  Then the following hold:
		\begin{enumerate}
		    \item $S$ has an Agler decomposition of the form
		\begin{align}  
	&	    I - S(z) S(w)^{*}  \notag \\
	& \quad = 
\int_{\widehat {\mathbb T}^{N}_{{\mathbb A}_{q}}}
 H_{[U], t}(z) \left( I_{\cX_{[U],t}} \otimes (I - 
 \Phi_{[U],t}(z) \Phi_{[U],t}(w)^{*} \right) H_{[U],t}(w)^{*} \, {\tt 
 d}\nu([U], t).
 \label{annAglerdecom} 
	\end{align}

\item There is a representation $\rho$ of $ C(\widehat{\mathbb T}^{N}_{{\mathbb 
		    A}_{q}}, \cL({\mathbb C}^{N}))$ on a Hilbert 
		    space $\cX$ and a unitary 
		    colligation matrix
$$
 \bU = \begin{bmatrix} A & B \\ C & D \end{bmatrix} \colon 
 \begin{bmatrix} \cX \\ {\mathbb C}^{N} \end{bmatrix} \to 
      \begin{bmatrix} \cX \\ {\mathbb C}^{N} \end{bmatrix}
$$
so that $S$ has the transfer-function realization
$$
  S(z) = D + C (I - \rho({\mathbb E}(z)) A)^{-1} \rho({\mathbb E}(z)) 
  B.
$$
\end{enumerate}
\end{corollary}

\begin{remark} \label{R:appealing}  {\em An appealing conjecture is 
    that the Agler decomposition \eqref{annAglerdecom} is {\em 
    minimal} in the sense of \cite[Section 5.1]{DM07} and 
    \cite[Section 3.6]{DP}.  
  } \end{remark}

 \subsection{The constrained Schur class over the unit disk}  
 \label{S:constrained}
 
 Following \cite{DPRS, BBtH}, we define the {\em constrained Hardy 
 space} $H^{\infty}_{1}$ over the unit disk ${\mathbb D}$
 to consist of bounded analytic functions $s$ on ${\mathbb D}$ such 
 that $s'(0) = 0$.  One can check that this is still an algebra. 
  In this section we identify a class of test 
 functions $\Psi^{N}_{1}$ for which the unit ball 
 $\overline{\cB}(H^{\infty})^{N \times N}$ of the algebra of $N 
 \times N$ matrices over $H^{\infty}_{1}$ (with norm equal to the 
 multiplier norm as multiplication operators on $(H^{2})^{N}$) can 
 best identified as the test-function Schur-Agler class 
 $\mathcal{SA}_{\Psi^{N}_{1}}({\mathbb C}^{N})$.

 The analysis parallels that of Section \ref{S:cR} for the Schur class 
 over a finitely connected planar domain.  One first identifies the 
 extreme points for the Herglotz class $\cH_{1}^{N}$ consisting of $N 
 \times N$ matrix-valued functions $F$ on ${\mathbb D}$ satisfying the 
 normalization $F(0) = I$ together with the side constraint $F'(0) = 
 0$.  Such functions are exactly the Cayley transforms
 $$
   F(z) = (I - S(z))^{-1} (I + S(z))
 $$
 of functions $S$ in the closed unit ball 
 $\overline{\cB}(H^{\infty}_{1})^{N \times N}$ of the  matrix-valued 
 constrained Hardy algebra $(H^{\infty}_{1})^{N \times N}$ subject to 
 the normalization $S(0) = 0$.  As is the case for any matrix-valued 
 Herglotz function on ${\mathbb D}$, there is a positive 
 matrix-valued measure $\mu$ on ${\mathbb T}$ so that $F$ has the 
 Herglotz representation
 $$
   F(z) = \int_{\mathbb T}  \frac{\zeta + z}{\zeta - z} \, {\tt 
   d}\mu(\zeta).
 $$
 The constraint that $F(0) = I_{N}$ is equivalent to $\mu({\mathbb 
 T}) = I_{N}$; following the terminology used in Section \ref{S:cR},
 we then say that  $\mu$ is an {\em $N \times N$ quantum probability 
 measure}.    The constraint that $F'(0)$ equals zero (i.e., that $F \in 
 \cH^{N}_{1}$) imposes the constraints on the measure $\mu$:
 $$
 F'(0) = \int_{{\mathbb T}} \zeta^{-1} {\tt d}\mu(\zeta) 
 =\int_{{\mathbb T}} \overline{\zeta} {\tt d}\mu(\zeta) = 0.
 $$
 Taking real and imaginary part then gives us two real constraints
 \begin{equation}   \label{muconstraint}
 \int_{{\mathbb T}} \R \zeta\, {\tt d} \mu(\zeta) = 0,   \quad
 \int_{{\mathbb T}} \I \zeta\, {\tt d} \mu(\zeta) = 0.
 \end{equation}
 We thus see that the convex set $\cH^{N}_{I}$ (the constrained 
 matrix-valued Herglotz class over ${\mathbb D}$) is affinely 
 equivalent to the convex set of measures
 $$
 \cC^{N}_{1} = \{ \mu: \mu = \text{\rm $N \times N$ quantum probability measure   such 
 that \eqref{muconstraint} holds}\}.
 $$
 This convex set of measures is compact in the weak-$*$ topology 
 (viewing complex $N \times N$ matrix-valued  measures as the dual space of
 ${\mathbb C}^{N}$-valued continuous functions on ${\mathbb T}$) and hence, by 
 the Kre\u{\i}n-Milman theorem, has extreme points.
 By the same general results from \cite{BG} leading to the the 
 identification of the set \eqref{HerglotzRextreme} of the normalized 
 Herglotz class $\cH(\cR)_{I}$ over the planar domain $\cR$, it 
 follows that the extreme points of $\cC^{N}_{1}$ can be described as 
 follows.  We let $\widehat \Theta^{N}$ consist of all pairs $(\bt, 
 \bw)$ where $\bt = (t_{1}, \dots, t_{n})$ is an $n$-tuple of points 
 on the unit circle ${\mathbb T}$ (with $1 \le n \le 3 N$) and $\bw = 
 (W_{1}, \dots, W_{n})$ is an $n$-tuple of $N \times N$ matrix 
 weights such that the following property holds: {\em
 ${\mathbf 0} = 0 \otimes I_{N}$ is in the interior 
 of the $C^{*}$-convex hull of $\bphi(\bt) \otimes I_{N}$}, where we 
 set
 $$
 \bphi(\bt) = \left\{ \begin{bmatrix} \R t_{1} \\ \I t_{1} 
\end{bmatrix}, \dots, \begin{bmatrix} \R t_{n} \\ \I t_{n} 
\end{bmatrix} \right\} \subset {\mathbb R}^{2}.
$$
{\em with a choice of matrix barycentric coordinates of ${\mathbf 0}$ 
with respect to $\bphi(\bt) \otimes I_{N}$ equal to $\{W_{1}, \dots, 
W_{n}\}$} (refer back to Section \ref{S:cR} for the definition of 
terms).  One consequence of the definitions is that, for any such 
$(\bt, \bw) = (t_{1}, \dots, t_{n}; W_{1}, \dots, W_{n})$ in 
$\widehat \Theta^{N}$, it holds that
\begin{equation}   \label{1system}
    \sum_{r=1}^{n} (\R t_{r}) W_{r} = 0, \quad
    \sum_{r=1}^{n} (\I t_{r}) W_{r} = 0.
 \end{equation}

Associated with each $(\bt, \bw) \in \widehat \Theta^{N}$ is a 
  holomorphic $N \times N$-matrix function on the unit disk given by
  $$
  F_{\bt, \bw}(z) = \sum_{r=1}^{n} \frac{t_{r} + z}{t_{r}-z} W_{r}.
  $$
  These functions are holomorphic on ${\mathbb D}$ with positive real 
  part, and moreover, as a consequence of \eqref{1system}, have the 
  property that $F_{\bt, \bw}'(0) = 0$.   In fact, it can be shown that the set of all such functions
  $\{ F_{\bt, \bw} \colon (\bt, \bw) \in {\mathbb T}^{N}_{1} \}$
  is exactly the set of extreme points for the normalized constrained 
  Herglotz class over ${\mathbb D}$, i.e., the class
  \begin{align*}
  (\cH^{N}_{1})_{I_{N}}: = &   \{ F \colon {\mathbb D} \to {\mathbb C}^{N \times N} 
  \colon F \text{ holomorphic, } \R F(z) \ge 0 \text{ for } z \in 
  {\mathbb D}, \\ & \quad  F(0) = I_{N}, \, F'(0) = 0 \}.
  \end{align*}
  By using Choquet theory it then follows that a general element $F$ 
  of $(\cH_{1}^{N})_{I_{N}}$ has an integral representation of the form
  $$
  F(z) = \int_{\widehat \Theta^{N}} F_{\bt, \bw}(z)\, {\tt d}\nu(\bt, 
  \bw)
  $$
  for some probability measure on $\widehat \Theta^{N}$.
  
  We note that $(\cH^{N}_{1})_{I_{N}}$ is exactly the Cayley transform of the 
  normalized constrained Schur class 
  \begin{align*}
      (\cS^{N}_{1})_{0} = & \{ S \colon {\mathbb D} \to {\mathbb C}^{N \times 
  N} \colon S \text{ holomorphic, } \| S(z) \| \le 1 \text{ for } z 
  \in {\mathbb D}, \\
  & \quad S(0) = 0, \, S'(0) = 0 \}, 
  \end{align*}
  i.e., 
  \begin{align*}
  & S \in (\cS^{N}_{1})_{0} \Leftrightarrow F: = (I - S)^{-1}(I+S) \in 
  (\cH^{N}_{1})_{I_{N}},  \\
  & F \in (\cH^{N}_{1})_{I_{N}} \Leftrightarrow S: = (F+I)^{-1} (F - 
  I) \in (\cS^{N}_{1})_{0}.
  \end{align*}
  In particular, for each $(\bt, \bw) \in {\mathbb T}^{N}_{1}$ we may
  define functions $S_{\bt, \bw} \in (\cS^{N}_{1})_{0}$ which in turn 
  leads us to the following collection of functions in $(\cS^{N}_{1})_{0}$:
  \begin{equation}   \label{PsiN1}
    \Psi^{N}_{1}: =\{ S_{\bt, \bw}(z) =  (F_{\bt, \bw}(z) +I)^{-1} (F_{\bt, \bw}(z) - 
    I) \colon (\bt, \bw) \in \widehat \Theta^{N}\}.
  \end{equation}
   Following the proof of Theorem 5.4 in \cite{BG} (the parallel 
  result for the matrix Schur class over  a planar domain $\cR$ in 
  place of $\cS^{N}_{1}$) then leads to the 
  integral Agler decomposition for the normalized constrained Schur 
  class:  {\em given $S \in (\cS^{N}_{1})_{0}$ there is a 
  function $((\bt, \bw), z) \mapsto H_{\bt, \bw}(z)$ on $\widehat 
  \Theta^{N} \times {\mathbb D}$, bounded and
  measurable in $(\bt, \bw) \in {\mathbb T}^{N}_{1}$ for each fixed 
  $z$, together with a probability measure $\nu$ on $\widehat 
  \Theta^{N}$,  so that}
  \begin{equation}   \label{1Aglerdecom}
  I - S(z) S(w)^{*} = 
  \int_{\widehat \Theta^{N}} H_{\bt, \bw}(z) \left( I - S_{\bt, \bw}(z) 
  S_{\bt, \bw}(w)^{*}  \right) H_{\bt, \bw}(w)^{*} \, {\tt d} \nu(\bt, \bw).
  \end{equation}
  If $S$ is in the strict constrained Schur class ($S \in 
  \overline{\cB}(H^{\infty}_{1})^{N \times N}$ with $\| S(0) \| < 
  1$), then  there is a choice of matrix M\"obius transformation on 
  the $N \times N$-matrix ball $T_{S(0)}$ so that $T_{S(0)}[S(z)]$ is 
  in the normalized constrained Schur class $(\cS^{N}_{1})$ (see 
  e.g.~\cite[Section 5]{BG}.  Using this one can see that functions 
  $S$ in the strict but unnormalized Schur class $\cS^{N}_{1}: = 
  \overline{\cB}(H^{\infty}_{1})^{N \times N}$ have the continuous 
  Agler decomposition \eqref{1Aglerdecom} as well.  
  
  Once this Agler 
  decomposition is in hand, by using the same techniques as used in 
  the proofs of Theorems \ref{T:matrixSchurR} (adaptations to the 
  matrix-valued setting of arguments in \cite{DM07} and \cite{DP}),  
  one can arrive at the following result. 
  
  \begin{theorem}   \label{T:1matrixSchur}
      With $\Psi^{N}_{1} \subset (\cS^{N}_{1})_{0}$ given by 
      \eqref{PsiN1}, we have the identity
      $$
      \overline{\cB}(H^{\infty}_{1})^{N \times N} = 
      \mathcal{SA}_{\Psi^{N}_{1}}.
      $$
  \end{theorem}
  
  There is also a dual pair of solution criteria for the 
  interpolation problem for the class $\cS^{N}_{1}$.  We first need 
  to introduce the generating set of admissible kernels for the class 
  $\cK_{\Psi^{N}_{1}}({\mathbb C}^{N})$ as follows.  For each isometric 
  $2N \times 1$ matrix, written as $\left[ 
  \begin{smallmatrix} \alpha \\ \beta \end{smallmatrix} \right]$ with 
      $\alpha$ and $\beta$ equal to  $N \times 1$  column vectors 
      satisfying $\alpha^{*} \alpha + \beta^{*} \beta = 1$, we 
      introduce the collection of $N \times N$-matrix kernel functions
     \begin{align}  
	 (\Psi^{N}_{1})^{0} = & \{
  K^{\alpha, \beta}(z,w) : =( \alpha z + z \beta) (\alpha^{*} + 
  \overline{w} \beta^{*}) + \frac{z^{2} \overline{w}^{2}}{1 - z 
  \overline{w}} I_{N} \colon   \notag \\
  & \quad \alpha, \beta \in {\mathbb C}^{N \times 
  1}, \, \alpha^{*} \alpha + \beta^{*} \beta = 1\}.
  \label{genkernels}
  \end{align}
  Then we have the following result.
  
    \begin{theorem} \label{T:1interpolation}
      Suppose that we are given an $N \times N$ matrix-valued 
      function $S_{0}$ on the subset ${\mathbb D}_{0}$ of the unit 
      disk ${\mathbb D}$.  Then the following are equivalent:
      
      \begin{enumerate}
	  \item There is a function $S$ in the restricted Schur class 
	  $\cS^{N}_{1}$ with $S|_{{\mathbb D}_{0}} = S_{0}$.
	  
	  \item There is a matrix-valued function $((\bt, \bw), z) 
	  \mapsto H_{\bt, \bw}(z)$ on $\widehat \Theta^{N} \times 
	  {\mathbb D}_{0}$, bounded and measurable in $(\bt, \bw)$ 
	  for each fixed $z \in {\mathbb D}_{0}$, together with a finite 
	  measure $\nu$ on $\widehat \Theta^{N}$, so that
$$
I - S_{0}(z) S_{0}(w)^{*} = \int_{\widehat \Theta^{N}} 
H_{\bt, \bw}(z) \left(I_{\cX_{\bt, \bw}} \otimes (I - S_{\bt, \bw}(z) S_{\bt, 
\bw}(w)^{*}) \right) H_{\bt, \bw}(w)^{*}\, {\tt d}\nu(\bt, \bw).
$$

\item  For each $2N \times 1$ isometric matrix $\left[ 
\begin{smallmatrix} \alpha \\ \beta \end{smallmatrix} \right]$ and 
    for each $Y  \colon {\mathbb D}_{0} \to {\mathbb C}^{N \times 
    N}$, the kernel
  \begin{equation}   \label{1-Pick}
      k(z,w) = \operatorname{tr} \left( Y(w)^{*}( I - S_{0}(w)^{*} S_{0}(z) 
)Y(z) ) K^{\alpha, \beta}(z,w) \right)
\end{equation}
(where $K^{\alpha, \beta}$ is given by \eqref{genkernels}) is a 
positive kernel on ${\mathbb D}_{0}$.
 \end{enumerate}
  \end{theorem}
  
  \begin{proof}  The proof parallels that of Theorem 
      \ref{T:Rinterpolation}.  To verify the equivalence of condition 
      (2) with existence of a solution of the interpolation problem, 
      use Theorem \ref{T:1matrixSchur} in combination with Theorem 
      \ref{T:SchurAgler}.  By Remark \ref{R:interpolation}, the 
      validity of condition (3) follows if we can verify that the 
      collection $(\Psi_{1}^{N})^{0}$ given by \eqref{genkernels} is a generating set for the 
      collection of admissible kernels $\cK_{\Psi^{N}_{1}}({\mathbb 
      C}^{N})$.  However, rather than doing this we use Theorem 
      1.3 from \cite{BBtH}.  As was the case for the Schur class 
      over a domain $\cR$, the form presented there is somewhat 
      different from the form \eqref{1-Pick} as presented here.  
      However, one can follow the argument in \cite{BBtH} and work 
      with right multiplication operators on $\cH(K^{\alpha, 
      \beta})_{{\mathbb C}^{N}}$ rather than left multiplication 
      operators on a left-sided tensor of the coefficient space with a 
      reproducing kernel Hilbert space of row-vector functions to 
      arrive at the form \eqref{1-Pick} as the solution criterion.
  \end{proof}

  \begin{remark} \label{R:1reduction} {\em As was observed in 
      connection with Corollary \ref{C:annulus}, the Schur-Agler 
      class $\mathcal{SA}_{\Psi}$ associated with a collection of test 
      functions $\Psi$ depends on the functions $\psi \in \Psi$ only 
      through the kernels $I - \psi(z) \psi(w)^{*}$.  Hence, for 
      $S_{\bt, \bw} $ in the test-function class $\Psi^{N}_{1}$ we 
      may define an equivalence relation $S_{\bt, \bw} \sim S_{\bt',\bw'}$ 
      when there is a unitary constant matrix $U$ so that $S_{\bt', 
      \bw'}(z) = S_{\bt, \bw}(z) U$.  To choose one representative 
      out of each equivalence class, we may normalize $S \in 
      \Psi^{N}_{1}$ so that $S(1) = I_{N}$.  This has the effect of 
      restricting the parameter $(\bt, \bw)$ in $\widehat \Theta^{N}$ 
      to those such that $1$ is one of the points in the set of 
      points $\bt = (1, t_{2}, \dots, t_{n})$ with associated weight 
      $W_{1}$ invertible; in this way we get a new smaller parameter 
      space $\Theta^{N}$.  Then we have 
      $\overline{\cB}(H^{\infty}_{1})^{N \times N} = 
      \mathcal{SA}_{\widetilde \Psi^{N}_{1}}$ where $\widetilde 
      \Psi^{N}_{1} = \{ S_{\bt, \bw} \colon (\bt, \bw) \in  
      \Theta^{N}\}$ is this restricted class of test functions.
      
      For the case $N=1$ (the scalar case), Theorem 
      \ref{T:1matrixSchur} is due to Dritschel-Pickering \cite{DP}.  
      In this case the parameter space 
      $\widehat \Theta^{1} =: \widehat \Theta$ can be described in 
      geometric terms as consisting of (1) triples of points on the unit 
      circle such that 0 is in the interior of the associated 
      triangle, with the weights then being the barycentric 
      coordinates of 0 with respect to this triangle, or (2) a pair 
      of antipodal points on the unit circle with weights then 
      necessarily $(\frac{1}{2}, \frac{1}{2})$.  When the reduction 
      described in the previous paragraph is carried out, one 
      restricts to triples of points $\bt = (1, t_{2}, t_{3})$ which 
      include $1$ and there is only one antipodal pair of points $(1, -1)$.
      These authors also show that this space $\Theta$ with its 
      natural topology is homeomorphic to the unit sphere. They also 
      show  that the collection $\widetilde \Psi^{1}_{1}$ is a minimal 
      collection of test functions for 
      $\overline{\cB}H^{\infty}_{1}$.  Whether 
      $\widetilde \Psi_{1}^{N}$ is a minimal collection of test functions for 
      $\overline{\cB}(H^{\infty}_{1})^{N \times N}$ in general we 
      leave as an open question.
      
      As we have seen, there is a dual issue of finding minimal 
      generating sets for admissible collections of kernels 
      $\cK_{\Psi}({\mathbb C}^{N})$, as well as finding small 
      generating sets for such $\cK_{\Psi}({\mathbb C}^{N})$. 
      In particular, it would be 
      interesting to see a direct proof that $(\Psi^{N}_{\cR})^{0}$ 
      in \eqref{KU} generates 
      $\cK_{\Psi^{N}_{\cR}}$ and that the set $(\Psi_{1}^{N})^{0}$ in \eqref{genkernels}  
      generates $\cK_{\Theta^{N}}({\mathbb C}^{N})$. We note that the 
      proofs of the interpolation results from \cite{Abrahamse, 
      Ball79, DPRS, BBtH} use the dual factorization approach (see 
      \cite{DH} for a unified setting); an independent proof of the 
      generating property for $(\Psi^{N}_{\cR})^{0}$ and  
      $(\Psi^{N}_{1})^{0}$  
      would mean that Theorem \ref{T:SchurAgler} gives an independent 
      proof of these interpolation results.
       }\end{remark}
       
       \begin{remark} \label{R1:generalizations}  {\em 
	   An alternative description of $H^{\infty}_{1}$ is 
	   ${\mathbb C} + z^{2}H^{\infty}$.  Many of the results 
	   concerning the space $H^{\infty}_{1}$ have been 
	   generalized to  more general algebras of the form ${\mathbb C} 
	   + B H^{\infty}$ where $B$ is a Blaschke product (see 
	   e.g.~\cite{Rag08}).
	   We believe that the results from \cite{BG} are 
	   sufficiently flexible to lead to test-function 
	   Schur-Agler-class characterizations of matrix-valued 
	   versions of these more general algebras as well.
	   }\end{remark}

\end{document}